\documentclass[a4paper]{amsart}

\usepackage{amsmath,amsthm,amssymb,enumerate}
\usepackage{epsfig}
\usepackage[utf8]{inputenc}
\usepackage[T1]{fontenc}
\usepackage{path}
\usepackage{ae,aecompl}
\usepackage{amsfonts}
\usepackage{amsxtra}
\usepackage{euscript,mathrsfs}
\usepackage{color}
\usepackage[left=3.4cm,right=3.4cm,top=4cm,bottom=4cm]{geometry}
\usepackage[citecolor=blue,colorlinks=true]{hyperref}
\allowdisplaybreaks

\newtheorem{assump}[equation]{Assumption}

\usepackage{enumitem}
\setenumerate{label={\rm (\alph{*})}}

\numberwithin{equation}{section}

\usepackage{amsfonts}
\usepackage{amsxtra}

\usepackage[frak,hyperref,theorem_section]{paper_diening}
\numberwithin{equation}{section}

\newcommand{\Div}{\divergence}

\newcommand{\R}{\mathbb R}

\newcommand{\px}{{p(\cdot)}}

\newcommand{\PP}{{\mathcal P}}

\newcommand{\PPln}{\mathcal{P}^{{\log}}}

\newcommand{\dd}{\mathrm d}
\newcommand{\dx}{\, \mathrm{d}x}
\newcommand{\dy}{\, \mathrm{d}y}

\newcommand{\ds}{\, \mathrm{d}\sigma}
\newcommand{\dt}{\, \mathrm{d}t}
\newcommand{\dxt}{\,\mathrm{d}x\, \mathrm{d}t}
\newcommand{\dxs}{\,\mathrm{d}x\, \mathrm{d}\sigma}
\newcommand{\dif}{\mathrm{d}}

\DeclareMathOperator{\diver}{div}

\begin{document}

\title[Space-time approximation of parabolic systems with variable growth]{Space-time approximation of parabolic systems with variable growth}

\author{Dominic Breit \& Prince Romeo Mensah}
\footnotetext[1]{Department of Mathematics, Heriot-Watt University, Edinburgh EH14 4AS, UK. (d.breit@hw.ac.uk, pm27@hw.ac.uk)}
\address{Department of Mathematics, Heriot-Watt University, Riccarton Edinburgh EH14 4AS, UK}
\email{}

%
%

\begin{abstract}
We study a parabolic system with $p(t,x)$-structure under Dirichlet boundary conditions. In particular, we deduce the optimal convergence rate for the error of the gradient of a finite element based space-time approximation. The error is measured in the quasi norm and the result holds if the exponent $p(t,x)$ is $(\alpha_t,\alpha_x)$-H\"{o}lder continuous.
\end{abstract}


\date{\today}

\maketitle

%
%
%
%
%
%
%
%
%
%

\section{Introduction}
Let $\Omega\subset\R^n$, $n\geq 2$,   $T>0$ be finite and assume that $\bff:Q\rightarrow\mathbb{R}^N$ is given where we have set $Q:=(0,T)\times\Omega$. We are interested in the parabolic $p(t,x)$-Laplace system
\begin{align}
\left\{\begin{array}{rc}
\partial_t \bfu=\Div\big((\kappa+|\nabla\bfu|)^{p(t,x)-2}\nabla\bfu\big)+\bff& \mbox{in $Q=(0,T)\times\Omega $,}\\
\bfu=0\qquad\quad& \mbox{ \,on $(0,T)\times\partial\Omega$,}\\
\bfu(0,\cdot)=\bfu_0\qquad& \mbox{in $\Omega$.}\end{array}\right.\label{eq:heat}
\end{align}
with $\kappa\geq0$ and $p:Q\rightarrow(1,\infty)$ a given continuous function. Problem \eqref{eq:heat} is motivated by a model for electro-rheological fluids \cite{RajR96,Ruz00}. These are smart
materials whose viscosity depends on the applied electric field. This is modelled -- roughly speaking -- via a dependence of the viscosity on the electric field similar to the constitutive relation
\begin{align}
\label{eq:stress}
\bfS(\nabla\bfu)=\bfS(t,x,\nabla\bfu)=(\kappa+|\nabla\bfu|)^{p(t,x)-2}\nabla\bfu
\end{align}
describing the main part in \eqref{eq:heat}. One can consider \eqref{eq:heat} as a model problem which one has to understand before turning to the full
generalized Navier--Stokes system.\\
In the last two decades, there has been a huge effort in understanding the functional analytic
set-up for problems involving variable exponents resulting for instance in the monograph \cite{DiHaHaRu}. Related PDE models have been studied both from a purely mathematical perspective (see, e.g., \cite{AceM05} and \cite{AcMiSe}) as well as in view of the model for electro-rheological fluids (see, e.g., \cite{AcMi} and \cite{DMS}). From the analytical point of view, problems such as \eqref{eq:heat} seem to be well-understood. However, as far as the numerical approximation of PDEs involving variable exponents is concerned, there are only a few available results.
Recently, a finite-element based approximation of the $p(x)$-Laplacian
\begin{align}\label{pxlaplace}
-\Div\big((\kappa+|\nabla\bfu|)^{p(x)-2}\nabla\bfu)\big)=\bff,
\end{align}
the elliptic counterpart of \eqref{eq:heat}, has been studied in \cite{BrDS2}. A key point
is the analysis of interpolation operators on generalized Lebesgue spaces
\begin{align}\label{eq:Lpx}
  L^{\px}(\Omega)=\Bigset{f\in
    L^1(\Omega):\,\,\int_{\Omega}\abs{f}^{\px}\dx < \infty}
\end{align}
with $p:\Omega \rightarrow(1,\infty)$.
The main difficulty here is the lack of Jensen's inequality in this setting (this is overcome by a modification of Diening's key estimate for generalized Lebesgue spaces \cite{Die04} under a certain continuity assumption on $p$) and an additional error term appears, cf. Theorem \ref{thm:app_V}.  Based on this, the error between the exact solution $\bfu$ to
\eqref{pxlaplace} and its finite element approximation $\bfu_h$ (where $h$ denotes the discretization parameter, see Section \ref{subsec:FEM} for more details) has been analysed in \cite{BrDS2}. The result is the convergence rate
\begin{align}\label{eq:2}
\|\bfF(\cdot,\nabla\bfu)-\bfF(\cdot,\nabla\bfu_h)\|_2^2\leq\,c\,\mathrm h^{2\alpha},
\end{align}
where $p$ is assumed the be $\alpha$-H\"older continuous with $0<\alpha\leq1$.
Here, the quantity $\bfF(x,\bfxi)=(\kappa+|\bfxi|)^{\frac{p(x)-2}{2}}\bfxi$
linearises equation \eqref{pxlaplace} in a certain sense and is typically used to study the regularity properties of solutions. The natural regularity which is required for
\eqref{eq:2} is
\begin{align}\label{reg:x}
\bfF(\cdot,\nabla \bfu)\in \big(W^{\alpha,2}(\Omega)\big)^{N\times n},
\end{align}
where $W^{\alpha,2}(\Omega)$ is the fractional Sobolev space with differentiability $\alpha$, see Section \ref{sec:fracSob} for details. Related results are given in \cite{DeMa}, where the two-dimensional subquadratic case is studied, and in \cite{BeBrDi}, where the finite element approximation of a steady model for electro-rheological fluids has been analysed. 
The error estimate \eqref{eq:2}
generalises earlier results for the $p$-Laplacian (where $p\in(1,\infty)$ is constant)
from \cite{DR}. We remark that quasi-norms such as the left-hand side in \eqref{eq:2} have been introduced in \cite{LiuBar93rem} to study the error in the numerical approximation of the $p$-Laplacian.\\
In view of the results for \eqref{pxlaplace} the next natural step is to study the numerical approximation of parabolic PDEs with variable exponents such as \eqref{eq:heat}. We aim at approximating equation
\eqref{eq:heat} by an implicit Euler scheme in time and finite elements in space.
If $\{0=t_0<\cdots<t_M=T\}$ is a uniform partition of $[0,T]$ with mesh size $\Delta t=T/M$, we approximate $\bfu(t_m)$ by $\bfu_{m,h}$, where $h$ refer to the discretisation in space (as in the elliptic case). The precise algorithm can be found in equation \eqref{tdiscr}. Under suitable conditions on the data and mesh size, see \eqref{assumpt}--\eqref{assumpt2} below, our main result, given in Theorem \ref{thm:4}, states the following error estimate
\begin{align}\label{eq:error0}
\begin{aligned}
\max_{1\leq m\leq M}\Vert\bfu(t_m,\cdot)-\bfu_{m,h}\Vert_2^2+\Delta t &\sum_{m=1}^M\Vert\bfF(t_m,\cdot,\nabla\bfu(t_m))-\bfF(t_m,\cdot, \nabla\bfu_{m,h})\Vert_2^2
\\
&\leq \,c\,\big( \mathrm{h}^{2\alpha_x}+(\Delta t)^{2\alpha_t}\big),
\end{aligned}
\end{align}
where  $\bfF(t,x,\bfxi)=(\kappa+|\bfxi|)^{\frac{p(t,x)-2}{2}}\bfxi$.
Here, $p:\overline{Q}\rightarrow(1,\infty)$ is assumed to be $\alpha_x$-H\"older-continuous with respect to the spatial variable and $\alpha_t$-H\"older-continuous with respect to time, where $\alpha_x\in(0,1]$ and $\alpha_t\in(\frac{1}{2},1]$. The proof of \eqref{eq:error0} requires a careful analysis of the time and space-dependence of the variable exponent $p$. The constant $c$ depends on the geometry of $\Omega$, on $p$, as well as some quantities involving $\bfu$ arising from the following regularity properties: 
\begin{align}\label{eq:regt1}
&\bfF(\cdot,\cdot, \nabla\bfu)\in \big(L^2(0,T;W^{\alpha_x,2}(\Omega)))\big)^{N\times n},\\
\label{eq:regt2}&\bfF(\cdot,\cdot, \nabla\bfu) \in \big(W^{\alpha_t,2}(0,T;L^2(\Omega))\big)^{N\times n},\\
\label{eq:regt3}&\bfu\in \big(L^\infty(0,T;W^{\alpha_x,2}(\Omega))\big)^N.
\end{align}
The relations \eqref{eq:regt1}--\eqref{eq:regt3} are the natural counterpart to the regularity from the elliptic case, cf. \eqref{reg:x}.
We give a proof of \eqref{eq:regt1}--\eqref{eq:regt3} in Theorem \ref{thm:quasispace} under natural assumptions on the data. Let us remark that the left-hand side of \eqref{eq:error0} is only defined provided the mapping
$$[0,T]\ni t\mapsto \bfF(t,\cdot,\nabla\bfu(t))\in \big(L^2(\Omega)\big)^{N\times n}$$
is continuous in time.
Given the regularity in \eqref{eq:regt2}, this holds if and only if $\alpha_t> \frac{1}{2}$ (cf. Theorem \ref{thm:sobboch}).  So, the estimate \eqref{eq:error0} must fail otherwise. This problem is not a consequence of the variable exponents but already appears for constant $p$ and even in the linear case $p=2$. In order to include the case $\alpha_t\leq \frac{1}{2}$ it is necessary to develop a new error estimator.
\\
In comparison with the $p$-Laplacian (where $p$ is constant), let us also mention that
\eqref{eq:error0} yields the same convergence rate (linear convergence with respect to both discretization parameters)
if $p$ is smooth enough (if $\alpha_x=\alpha_t=1$). Convergence rates for the $p$-Laplacian are given in \cite{DER} (see also \cite{BaLi2} for the case $p\geq2$ and \cite{EbLi} for $p>\frac{3n}{n+2}$).
The only result for the numerical approximation of parabolic problems involving variable exponents is \cite{CP}. The authors of \cite{CP} study the full generalized Navier--Stokes equation under the assumption that $p>\frac{3n+2}{n+2}$. This is of course much more complicated than \eqref{eq:heat} but only leads to a much weaker result: convergence of a subsequence without any rates. Convergence rates for the numerical approximation of parabolic problems such as \eqref{eq:heat} are not known at all.\\
Our paper is organized as follows. In Section \ref{sec:prel} we collect some basic material regarding function spaces (generalized Lebesgue spaces and fractional Sobolev spaces) as well as finite elements. 
Our main result, the error estimate \eqref{eq:error0}, is stated and proved in Section \ref{sec:error}. In Section \ref{sec:reg} we analyse the regularity of solutions to \eqref{eq:heat} and prove \eqref{eq:regt1}--\eqref{eq:regt3}. Finally we collect some elementary algebraic relations for Young functions in the appendix.

\begin{remark}
All results of this paper remain valid under more general assumptions on the operator $\bfS$ for which \eqref{eq:stress} is the model case. In fact, it is sufficient to assume that $\bfS:\overline{Q}\times \R^{N\times n}$ is continuous and satisfies
\begin{align}
\lambda (\kappa+|\bfeta|)^{p(t,x)-2}|\bfxi|^2&\leq D_\bfeta\bfS(t,x,\bfeta)(\bfxi,\bfxi)\leq\,\Lambda (\kappa+|\bfeta|)^{p(t,x)-2}|\bfxi|^2,\label{growth1}\\
|\bfS(t,x,\bfeta)-\bfS(s,y,\bfeta)|&\leq\,c\,(|t-s|^{\alpha_t}+|x-y|^{\alpha_x})\big(1+|\ln(\kappa+|\bfeta|)|\big)
\label{growth2}\\
&\quad\times \Big( (\kappa+|\bfeta|)^{p(t,x)-2} + (\kappa+|\bfeta|)^{ p(s,y)-2} \Big) |\bfeta| ,\nonumber
\end{align}
for all $(t,x),(s,y)\in \overline{Q}$ and all $\bfeta,\bfxi\in\R^{N\times n}$ with positive constants $\lambda,\Lambda,c$ and $\kappa\geq0$.
In particular, modulus dependence of $\bfS$ is not necessary. A detailed proof that
equation \eqref{eq:stress} implies \eqref{growth1} and \eqref{growth2} is given in Lemma \ref{lemma:growth} in the appendix.
\end{remark}

\section{Preliminaries}
\subsection{Variable exponent spaces}
\label{sec:prel}
\noindent
Let us introduce the variable exponent Lebesgue space $L^\px(\mathcal O)$, where $\mathcal O\subset \mathbb R^M$ for $M\geq2$, either stands for the domain $\Omega\subset\mathbb R^n$ or for the parabolic cylinder $Q\subset\R^{n+1}$.  We use the
same notation as in the recent book~\cite{DiHaHaRu}.  We define
$\PP(\mathbb R^M)$ to consist of all measurable functions $p\,:\, \mathbb R^M \to
[1,\infty]$ (called variable exponents).  For $p \in\PP(\mathbb R^M)$ we define
$p^- := \essinf p$ and $p^+:= \esssup p$.
For $p \in \PP(\mathbb R^M)$, the generalized Lebesgue space $L^{\px}(\mathcal O)$ is
defined as
\begin{align*}
  L^\px(\mathcal O):=\biggset{f \in L^1(\mathcal O) \,:\,
    \norm{f}_{L^\px(\mathcal O)} < \infty},
\end{align*}
where
\begin{align*}
  \norm{f}_{\px}:=\norm{f}_{L^\px(\mathcal O)}:=
  \inf\biggset{\lambda>0\,:\,\int_{\mathcal O} \biggabs{
      \frac{f(x)}{\lambda}}^{p(x)} \,dx \leq 1}.
\end{align*}
Many properties can be transferred from the classical $L^p$-spaces. In particular we have the following.
\begin{itemize}
\item $L^\px(\mathcal O)$ is uniformly convex and reflexive provided $1<p^{-}\leq p^{+}<\infty$.
\item $C^\infty_0(\mathcal O)$ is dense in $L^\px(\mathcal O)$ provided $p^+<\infty$.
\item The dual of $L^{\px}(\mathcal O)$ is $L^{p'(\cdot)}(\mathcal O)$ provided $p^+<\infty$. Here the dual exponent $p'(\cdot)$ is computed pointwise by $p'(x)=p(x)/(p(x)-1)$.
\item H\"older's inequality holds, i.e. we have
\begin{align*}
\|fg\|_1\leq\,2\|f\|_{\px}\|g\|_{p'(\cdot)}
\end{align*}
for all $f\in L^{\px}$ and $g\in L^{p'(\cdot)}$.
\end{itemize} 
%
Of course we can define Sobolev spaces with variable exponents in a natural way by setting
\begin{align*}
W^{1,\px}(\mathcal O):=\set{v\in W^{1,1}(\mathcal O)\, : \,v\in L^{\px}(\mathcal O),\,\,\nabla v\in (L^{\px}(\mathcal O))^n}.
\end{align*}
In order to study more advanced properties of $L^{\px}$, some mild regularity is needed.
We say that a function $g \colon \mathbb R^M \to \setR$ is {\em
  $\log$-H{\"o}lder continuous} on $\mathcal O$ if there exist constants
$c \geq 0$ and $g _\infty \in \setR$ such that
\begin{align*}
  \abs{g (x)-g (y)} &\leq \frac{c}{\log
    (e+1/\abs{x-y})}\quad\text{and}\quad \abs{g (x) - g _\infty}
  \leq \frac{c}{\log(e + \abs{x})},
\end{align*}
for all $x\not=y\in \mathcal O$. The first condition describes the so called
local $\log$-H{\"o}lder continuity and the second, the decay condition.
The smallest such constant~$c$ is the $\log$-H{\"o}lder constant
of~$g$. We define $\PPln(\mathbb R^M)$ to consist of those exponents $p\in
\PP(\mathbb R^M)$ for which $\frac{1}{p} \,:\, \mathbb R^M \to [0,1]$ is $\log$-H{\"o}lder
continuous. By $p_\infty$, we denote the limit of~$p$ at infinity,
which exists for $p \in \PPln(\mathbb R^M)$.  If $p \in \PP(\mathbb R^M)$ is bounded, then $p
\in \PPln(\mathbb R^M)$ is equivalent to the $\log$-H{\"o}lder continuity of $p$.
However, working with $\frac 1p$ gives better control of the constants
especially in the context of averages and maximal functions.
Therefore, we define $c_{\log}(p)$ as the $\log$-H{\"o}lder constant
of~$1/p$. Expressed in~$p$, we have for all $x,y \in \mathbb R^M$
\begin{align*}
  \abs{p(x) - p(y)} \leq \frac{(p^+)^2 c_{\log}(p)}{\log
    (e+1/\abs{x-y})} \quad \text{and }\quad \abs{p(x) - p_\infty}
  \leq \frac{(p^+)^2 c_{\log}(p)}{\log(e + \abs{x})}.
\end{align*}
It also turns out that  $C^\infty(\overline{\mathcal O})$ is dense in $W^{1,\px}(\mathcal O)$ provided $p \in \PPln(\R^M)$ with $p^+<\infty$. In this case we can define boundary values in the natural way. In particular, $W^{1,\px}_0(\mathcal O)$ is defined as the closure of $C^\infty_0(\mathcal O)$ in $W^{1,\px}(\mathcal O)$. We conclude this subsection by an elemenatry lemma which allows to replace $p(x)$ by $p(y)$ under certain assumptions.
\begin{lemma}[Lemma 2.1 in~\cite{BrDS2}]
  \label{lem:pxpy}  
  Let $p \in \PPln(\mathbb R^M)$ with $p^+<\infty$ and $\ell>0$. Then for every
  cube $K \subset \mathbb R^M$ with $|K|\leq 1$, $\kappa\in [0,1]$ and $t\geq
  0$ such that $\abs{K}^\ell \leq \rho \leq \abs{K}^{-\ell}$, we have that 
  \begin{align*}
    (\kappa + \rho)^{p(x)-p(y)} \leq c,
  \end{align*}
  for all $x,y \in K$. The constant depends on $c_{\log}(p)\,,\ell$ and
  $p^+$.
\end{lemma}

\subsection{Fractional Sobolev spaces}
\label{sec:fracSob}
In this subsection we collect basic material on
fractional Sobolev spaces. We define for $p\in(1,\infty)$ and $\alpha\in(0,1)$
the norm
\begin{align*}
\|u\|_{W^{\alpha,p}(\Omega)}:= \left(\|u\|^p_{L^{p}(\Omega)}+\int_\Omega\int_\Omega\frac{|u(y)-u(x)|^p}{|y-x|^{n+\alpha p}}\,\dd x\,\dd y \right)^\frac{1}{p}.
\end{align*}
The space $W^{\alpha,p}(\Omega)$ is now defined as the subspace of $L^{p}(\Omega)$
consisting of the functions having finite $W^{\alpha,p}(\Omega)$-norm. Subsequently in our computations, we will use the following notations to represent the seminorm 
$$
\left[  u \right]^p_{W^{\alpha,p}(\Omega)} :=
\int_\Omega\int_\Omega\frac{|u(y)-u(x)|^p}{|y-x|^{n+\alpha p}}\,\dd x\,\dd y
$$
for the fractional derivative in space. 
In the following two lemmas we recall the fundamental embedding properties for fractional order Sobolev spaces. (For the proof we refer to  \cite[Chapter 7]{Ad}, for Lemma \ref{le2} see also \cite[Lemma 2.5]{KrMi}).
\begin{lemma}\label{lema}
Let $\Omega\subset\R^n$ be a  bounded Lipschitz domain, $1< p<\infty$ and  $\alpha\in\big(0,\frac{n}{p})$.
The embedding $$W^{\alpha,p}(\Omega)\hookrightarrow L^{\frac{np}{n-\alpha p}}(\Omega)$$ is continuous.
\end{lemma}
\begin{lemma}\label{le2}
Let $\Omega\subset\R^n$ be a bounded Lipschitz domain and $q\in(1,\infty)$. Let $w\in
L^{q}(\Omega)$ and for some $\beta\in(0,1]$, $M>0$ and an open set $\tilde\Omega\Subset\Omega$, assume we have
$$
\sum_{\gamma=1}^{n}\int_{\tilde\Omega}|w(x+he_\gamma)-w(x)|^{q}\ dx\leq
M^{q} |h|^{q\beta}
$$
for every $h$ with $|h|<\mathrm{dist}(\tilde\Omega,\partial\Omega)$. Then $w\in
W^{k,q}_{loc}(\tilde\Omega)$ for every $k\in(0,\beta)$ and for each open set $A\Subset\tilde\Omega$ there is a constant $c=c(k,\mathrm{dist}(A,\partial\tilde\Omega))$, independent of $M$ and $w$, such that
$$
\|w\|_{W^{k,q}(A)}\leq
c\left(M+\|w\|_{L^{q}(\Omega)}\right).
$$
\end{lemma}
A very useful tool is Poincar\'e's inequality (see, for instance, \cite[Lemma 5.6]{GueErn15}).
\begin{lemma}\label{lem:poincare}
Let $\alpha\in(0,1]$ and $p\in[1,\infty)$. Then we have
\begin{align}\label{in:crucial0}
\int_\Omega |u-(u)_{\Omega}|^p\dx\leq\,c\,\mathrm{diam}(\Omega)^{\alpha p}\left[u\right]_{W^{\alpha,p}(\Omega)}^p
\end{align}
for all $u\in W^{\alpha,p}(\Omega)$, where $c$ does not depend on $\mathrm{diam}(\Omega)$.
Here, we denote $(u)_{\Omega}=\dashint_\Omega u\dx$.
\end{lemma}

Similarly, we can define fractional derivatives in time for functions $u:(0,T)\rightarrow X$, where $(X,\|\cdot\|_X)$ is a separable Banach space.
We define for $p\in(1,\infty)$ and $\alpha\in(0,1)$
the norm
\begin{align*}
\|u\|_{W^{\alpha,p}(0,T;X)}:=\left( \|u\|^p_{L^{p}(0,T;X)}+\int_0^T\int_0^T\frac{\|u(\sigma_1)-u(\sigma_2))\|_{X}^p}{|\sigma_1-\sigma_2|^{1+\alpha p}}\,\dd\sigma_1\,\dd\sigma_2 \right)^\frac{1}{p}.
\end{align*}
The space $W^{\alpha,p}(0,T;X)$ is now defined as the subspace of the Bochner space $L^{p}(0,T;X)$
consisting of the functions having finite $W^{\alpha,p}(0,T;X)$-norm.
Analogously to the fractional derivative in space we write
\begin{align*}
\left[u\right]^p_{W^{\alpha,p}(0,T;X)}:=\int_0^T\int_0^T\frac{\|u(\sigma_1)-u(\sigma_2))\|_{X}^p}{|\sigma_1-\sigma_2|^{1+\alpha p}}\,\dd\sigma_1\,\dd\sigma_2 
\end{align*}
in the time dependent case.
The following fractional variant of Sobolev's Theorem holds.
\begin{lemma}\label{thm:sobboch}
Let $X$ be a separable Banach space, $1< p<\infty$ and  $\alpha\in\big(\frac{1}{p},1\big)$.
The embedding $$W^{\alpha,p}(0,T;X)\hookrightarrow C^{\alpha-\frac{1}{p}}([0,T];X)$$ is continuous.
\end{lemma}
Based on Lemma \ref{thm:sobboch} we can proof the following corollary which will turn out to be crucial for our analysis.
\begin{corollary}\label{cor:crucial}
Under the assumptions of Lemma  \ref{thm:sobboch} we have
\begin{align}\label{in:crucial}
\int_0^T\|u(t)-u(T)\|_X^p\dt\leq\,cT^{\alpha p}\left[u\right]_{W^{\alpha,p}(0,T;X)}^p
\end{align}
for all $u\in W^{\alpha,p}(0,T;X)$, where $c$ does not depend on $T$.
\end{corollary}
\begin{remark}
We can replace $u(T)$ on the left-hand side of \ref{in:crucial} by $u(t_0)$ for any $t_0\in[0,T]$.
\end{remark}
\begin{remark}
Inequality \ref{in:crucial} holds if and only if $\alpha> \frac{1}{p}$.
This can be seen from the optimality of the embedding in Lemma  \ref{thm:sobboch}. If $\alpha\leq \frac{1}{p}$, functions from $W^{\alpha,p}(0,T;X)$ are not necessarily continuous in time. Consequently, the point evaluation $u(T)$ is not well-defined.
\end{remark}
\begin{proof}
Without loss of generality we assume that $T=1$. The general case follows by scaling. So, we are going to prove that
\begin{align}\label{in:crucial'}
\int_0^1\|u(t)\|_X^p\dt\leq\,C\left[u\right]_{W^{\alpha,p}(0,1;X)}^p
\end{align}
for all $u\in W^{\alpha,p}(0,1;X)$ with $u(1)=0$. We argue by contradiction.
Assume that \eqref{in:crucial'} is false. Then there is a sequence $(u_k)\in W^{\alpha,p}(0,1;X)$ with $\int_0^1\|u_k(t)\|_X^p\dt=1$, $u_k(1)=0$ and 
\begin{align}\label{in:crucial''}
\left[u_k\right]_{W^{\alpha,p}(0,1;X)}^p\leq\frac{1}{k}
\end{align}
for all $k\in\mathbb N$. Consequently, $(u_k)$ is bounded in $W^{\alpha,p}(0,1;X)$. So, there is a subsequence (not relabeled) such that $u_k\rightharpoonup u$ in $W^{\alpha,p}(0,1;X)$. By Lemma \ref{thm:sobboch} and the Arzel\`a-Ascoli Theorem, the embedding $W^{\alpha,p}(0,1;X)\hookrightarrow C( [0,1];X)$ is compact such that $u_k\rightarrow u$ uniformly. As a consequence, we have
$u(1)=0$ and $\int_0^1\|u(t)\|_X^p\dt=1$. On the other hand, \eqref{in:crucial''} yields
$\left[u\right]_{W^{\alpha,p}(0,1;X)}=0$ so that $u$ is constant in time. This is obviously a contradiction. 
\end{proof}

The following lemma combines fractional derivatives in space and time. It is a special case of a general interpolation result (see \cite[Thm. 3.1]{Am}).
\begin{lemma}\label{lem:inter}
Let $\alpha_t,\alpha_x\in(0,1]$ and $\theta\in(0,1)$.
The embedding
\begin{align*}
W^{\alpha_t,2}(0,T;L^2(\Omega))\cap L^2(0,T;W^{\alpha_x,2}(\Omega))\hookrightarrow W^{\theta\alpha_t,2}(0,T;W^{(1-\theta)\alpha_x,2}(\Omega))
\end{align*}
is continuous.
\end{lemma}

\subsection{Weak solutions}
In order to treat parabolic problems we define the function spaces
\begin{align*}
\mathcal V^{p(\cdot)}(Q):=\bigg\{u\in L^1(0,T;W^{1,1}(\Omega)):\,\,\int_Q |\nabla u|^{p(\cdot)}\dxt<\infty\bigg\},\\
\mathcal V^{p(\cdot)}_0(Q):=\bigg\{u\in L^1(0,T;W^{1,1}_0(\Omega)):\,\,\int_Q |\nabla u|^{p(\cdot)}\dxt<\infty\bigg\},
\end{align*}
see \cite{DNR}, in which weak solutions to \eqref{eq:heat} are located. 
\begin{definition}\label{def:weak}
Assume that $\bff=\Div\bfG $ with $\bfG \in \big(L^1(Q)\big)^{N\times n}$.
We call a function $\bfu\in C([0,T];(L^2(\Omega))^N)\cap (\mathcal V_0^{p(\cdot)}(Q))^N$ a weak solution to \eqref{eq:heat} iff
\begin{equation}
\begin{aligned}\label{eq:heatweak1}
\int_\Omega\bfu(t,x)\cdot\bfvarphi(x)\dx&=\int_\Omega \bfu_0(x)\cdot\bfvarphi(x)\dx-\int_0^t\int_\Omega \bfS(\sigma,x,\nabla\bfu(\sigma,x)):\nabla\bfvarphi(x)\dxs
\\
&-\int_0^t\int_\Omega \bfG (\sigma,x) : \nabla\bfvarphi(x)\dxs
\end{aligned}
\end{equation}
for all $\bfvarphi\in (C_0^\infty(\Omega))^N$ and all $t\in[0,T]$. 
\end{definition}
\begin{remark}
\begin{enumerate}
\item The formulation in \eqref{eq:heatweak1} is equivalent to 
\begin{equation}
\begin{aligned}\label{eq:heatweak1'}
\int_\Omega\bfu(t,x)\cdot\bfvarphi(x)\dx&=\int_\Omega \bfu(s,x)\cdot\bfvarphi(x)\dx-\int_s^t\int_\Omega \bfS(\sigma,x,\nabla\bfu(\sigma,x)):\nabla\bfvarphi(x)\dxs
\\
&-\int_s^t\int_\Omega \bfG(\sigma,x): \nabla\bfvarphi(x)\dxs
\end{aligned}
\end{equation}
for all $\bfvarphi\in (W_0^{1,\overline p_{(s,t)}(\cdot)}(\Omega))^N$ (with $\overline p_{(s,t)}(x):=\sup_{s\leq\sigma\leq t}p(\sigma,x)$) and all $s,t\in[0,T]$ with $s<t$,
and to
\begin{equation}
\begin{aligned}\label{eq:heatweak1''}
\int_Q\bfu(t,x)\cdot\partial_t\bfvarphi(t,x)\dxt&= \int_Q \bfS(t,x,\nabla\bfu(t,x)):\nabla\bfvarphi(t,x)\dxt
\\
&-\int_\Omega \bfu_0(x)\cdot\bfvarphi(0,x)\dx+\int_Q \bfG(t,x):\nabla\bfvarphi(t,x)\dxt
\end{aligned}
\end{equation}
for all $\bfvarphi\in (C_0^\infty([0,T)\times\Omega))^N$.
\item If $\bff\in \big(L^1(Q)\big)^N$, we can define $\bfG =\nabla\Delta^{-1}\big(\chi_{\Omega}\bff\big)$, where $\Delta^{-1}$ is the solution operator of the Laplace on $\mathbb R^n$,  and obtain
$$\int_0^t\int_\Omega \bff(\sigma,x)\cdot\bfvarphi(x)\dxs=-\int_0^t\int_\Omega \bfG (\sigma,x) : \nabla\bfvarphi(x)\dxs$$
for all $\bfvarphi\in (C_0^\infty(\Omega))^N$ and all $t\in[0,T]$.
\item A weak solution to \eqref{eq:heat} in the sense of Definition \ref{def:weak}
exists if $p\in\PPln(Q)$ with $p^->1$ and $p^+<\infty$ and $\bfG \in \big(L^{p'(\cdot)}(Q)\big)^{N\times n}$. Although the framework of Bochner spaces is not available in this setting, this can be shown by standard methods from monotone operator theory, see \cite{DNR}.
\end{enumerate}
\end{remark}

\subsection{Finite elements}
\label{subsec:FEM}

Let $\Omega \subset \setR^n$ be a connected, open
domain with polyhedral boundary. We assume that $\partial \Omega$ is
Lipschitz continuous. For an open, bounded (non-empty) set $U \subset
\setR^n$ we denote
by $h_U$ the diameter of $U$, and by $\rho_U$ the supremum of the
diameters of inscribed balls.  
We denote by $\mathscr{T}_h$, the simplicial subdivision of $\Omega$ where
\begin{align*}
  \mathrm{h} &= \max_{\mathcal S \in \mathscr{T}_h} h_\mathcal S.
\end{align*}
We assume that $\mathscr{T}_h$ is
non-degenerate, that is
\begin{align}
  \label{eq:nondeg}
  \max_{\mathcal S \in \mathscr{T}_h} \frac{h_\mathcal S}{\rho_\mathcal S} \leq \gamma_0.
\end{align}
For $\mathcal S \in \mathscr{T}_h$ we define the set of neighbors $N_\mathcal S$ and
the neighborhood $\mathcal M_\mathcal S$ by
\begin{align*}
  N_\mathcal S &:= \set{\mathcal S' \in \mathscr{T}_h \,:\, \overline{\mathcal S'} \cap
    \overline{\mathcal S} \not= \emptyset},
  \quad
  \mathcal M_\mathcal S := \text{interior} \bigcup_{\mathcal S' \in N_\mathcal S} \overline{\mathcal S'},
\end{align*}
respectively. Note that for all $\mathcal S,\mathcal S' \in \mathscr{T}_h$: $\mathcal S' \subset
\overline{\mathcal M_{\mathcal S}} \Leftrightarrow \mathcal S \subset \overline{\mathcal M_{\mathcal S'}}
\Leftrightarrow \overline{\mathcal S} \cap \overline{\mathcal S'} \not=\emptyset$. Due
to our assumption on $\Omega$, the $\mathcal M_\mathcal S$ are connected, open bounded sets. 

It is easy to see that the non-degeneracy~\eqref{eq:nondeg} of
$\mathscr{T}_h$ implies the following properties, where the constants
are independent of $\mathrm{h}$:
\begin{enumerate}
\item \label{mesh:SK} $\abs{\mathcal M_\mathcal S} \sim \abs{\mathcal S}$ for all $\mathcal S \in
  \mathscr{T}_h$.
\item \label{mesh:NK} There exists $m_1 \in \setN$ such that $\# N_\mathcal S
  \leq m_1$ for all $\mathcal S \in
  \mathscr{T}_h$.
\end{enumerate}
  
For $\Omega\subset \setR ^n$ and $\ell\in \setN _0$ we denote by
$\mathscr{P}_\ell(\Omega)$ the polynomials on $\Omega$ of degree less than or equal
to $\ell$. Moreover, we set $\mathscr{P}_{-1}(\Omega):=\set {0}$. Let us
characterize the finite element space $V_h$ as
\begin{align}\label{def:Vh}
  V_h &:= \set{v \in W^{1,1}(\Omega)\,:\, v|_{\mathcal S}
    \in \mathscr{P}_1(\mathcal S)\,\,\forall \mathcal S\in \mathscr T_h}.
\end{align}
%

We now state an assumption on an interpolation operator between the
continuous and discrete function spaces (satisfied e.g. by the Scott-Zhang
operator~\cite{ScoZha90}).  More
precisely, we assume the following.\footnote{We denote by $\dashint_A f\dx=|A|^{-1}\int_Af\dx$ the mean value of a integrable function $f$ over the set $A$.}

\begin{assump}
  \label{ass:intop}
There is $\Pi_h \,:\, (L^{1}(\Omega))^N \to
  (V_h)^N$ linear such that the following holds.
  \begin{enumerate}
  \item There holds uniformly in $\mathcal S \in
    \mathscr{T}_h$ and $\bfv \in (L^{1}(\Omega))^N$
  \begin{align}
    \label{eq:stab}
 \dashint_{\mathcal S} \abs{\Pi_h \bfv}\dx &\leq
    \,c 	 \dashint_{\mathcal M_\mathcal S} \abs{
      \bfv}\dx.
  \end{align}
\item For all $\bfv \in (\mathscr{P}_{1}(\Omega))^N$ we have
  \begin{align}
    \label{eq:proj}
    \Pi_h \bfv &= \bfv.
  \end{align}
  \end{enumerate}
\end{assump}
It is well-known that Assumption \ref{ass:intop} implies $L^p$-stability of $\Pi_h$, i.e., we have
  \begin{align}
    \label{eq:stabp}
 \dashint_{\mathcal S} \abs{\Pi_h \bfv}^p\dx &\leq
    \,c(p) 	 \dashint_{\mathcal M_\mathcal S} \abs{
      \bfv}^p\dx
\end{align}
for all $\bfv\in \big(L^p(\mathcal M_\mathcal S)\big)^N$. This implies immediately a stability result in fractional Sobolev spaces. In fact, we have for all $\bfv\in \big(W^{\alpha,p}(\mathcal M_\mathcal S)\big)^N$, $p\in[1,\infty)$ and $\alpha\in(0,1]$,
  \begin{align}
    \label{eq:stab''}
\begin{aligned}
\dashint_\mathcal{S} \left\vert  \bfv-\Pi_h\bfv\right\vert^p\dx &\leq\,c\dashint_\mathcal{S} \left\vert  \bfv-(\bfv)_{\mathcal{M}_\mathcal{S}}\right\vert^p\dx+c\dashint_\mathcal{S} \left\vert  \Pi_h\big(\bfv-(\bfv)_{\mathcal{M}_\mathcal{S}}\big)\right\vert^p\dx\\
&\leq\,c\dashint_{\mathcal{M}_\mathcal{S}} \left\vert  \bfv-(\bfv)_{\mathcal{M}_\mathcal{S}}\right\vert^p\dx\leq
    \,c\, h_\mathcal S^{2\alpha-n}  \left[
      \bfv\right]^2_{W^{\alpha,2}(\mathcal{M}_\mathcal{S})} 
\end{aligned}
  \end{align}
using \eqref{eq:proj}, \eqref{eq:stabp} and Lemma \ref{lem:poincare}.
Note that the constant $c$ in \eqref{eq:stab''} does not depend on $h_{\mathcal S}$.

The following crucial estimate follows from \cite[Lemma 4.2]{BrDS2}.
\begin{lemma}
  \label{thm:app_V}
Let $p\in \PPln(Q)$ with $p^->1$ and $p^+<\infty.$
Let $\Pi_h$ satisfy Assumption~\ref{ass:intop}. Let
  $\bfv \in (W^{1,p(\cdot,\cdot)}(\Omega))^N$, then for all $\mathcal{S} \in \mathcal{T}_h$ and all $\bfQ\in\mathbb{R}^{N\times n}$ with
    \begin{align*}
  \vert\bfQ\vert  +\dashint_{\mathcal{M}_\mathcal{S}}\vert\nabla\bfv\vert\dx  \leq c\, \max\{1,\vert \mathcal{S}\vert^{-1}\}
  \end{align*}
 it holds that
  \begin{align}
    \label{eq:app_V1}
    \dashint_{\mathcal S} \bigabs{\bfF (t,\cdot, \nabla \bfv) - \bfF (t,\cdot,\nabla \Pi_h
      \bfv)}^2 \dx &\leq c\, 
    \dashint_{\mathcal M_\mathcal S} \bigabs{\bfF(t,\cdot,\nabla \bfv) - \bfF(t,\cdot,\bfQ)}^2 \dx + ch_\mathcal{S}^2, 
  \end{align}
  uniformly in $t\in I$ with $c$ depending only on $p$ and $\gamma_0$.
\end{lemma}

\begin{corollary}\label{cor:app_V}
Let the assumption of Lemma \ref{thm:app_V} be valid. If $\bfF(t,\cdot,\nabla \bfv)
  \in (W^{\alpha,2}(\Omega))^{N\times n}$ for a.e. $t\in I$ then there is $\bfQ=\bfQ(t)\in\mathbb{R}^{N\times n}$ such that
  \begin{align}
    \label{eq:app_V2}
\begin{aligned}
    \dashint_{\mathcal M_\mathcal S} \bigabs{\bfF (t,\cdot,\nabla \bfv) - \bfF (t,\cdot,\bfQ)}^2 \dx &\leq c\, h_\mathcal S^{2\alpha}\,
    \left[\bfF(t,\cdot,\nabla \bfv)\right]^2_{W^{\alpha,2}(\mathcal{M}_\mathcal{S})}\\&+c\, h_\mathcal S^{2\alpha}\bigg(\dashint_{\mathcal M_{\mathcal S}}\left[\ln(\kappa+|\nabla\bfv|)\right]^2(\kappa+|\nabla\bfv|)^{p(t,\cdot)}\dx + 1\bigg)
\end{aligned}
  \end{align}
uniformly in $t\in I$. Moreover, if f $\bfF(t,\cdot,\nabla \bfv)
  \in L^\infty(0,T;(W^{\alpha,2}(\Omega))^{N\times n})$, we have
\begin{align}\label{eq:2808b}
\dashint_{\mathcal M_{\mathcal S}}\ln(\kappa+|\bfQ|)^2(\kappa+|\bfQ|)^{p(t,\cdot)}\dx\leq\,c\bigg(\dashint_{\mathcal M_{\mathcal S}}\left[\ln(\kappa+|\nabla\bfv|)\right]^2(\kappa+|\nabla\bfv|)^{p(t,\cdot)}+1\bigg)
\end{align}
uniformly in $t\in I$.
\end{corollary}
\begin{proof}
We consider the special choice of $\bfQ$ in Lemma \ref{thm:app_V} for which
\begin{align}
\label{eq:Q}
 \dashint_{\mathcal{M}_\mathcal S} \bfF (t,\cdot, \bfQ)  \dx 
 =  \dashint_{\mathcal{M}_\mathcal S} \bfF (t,\cdot, \nabla \bfv)  \dx
\end{align}
for fixed $t\in I$. Such a $\bfQ$ exists provided the function 
\begin{align*}
\lambda(\bfQ) :=  \dashint_{\mathcal{M}_\mathcal S} \bfF (t,\cdot, \bfQ)  \dx 
=
 \dashint_{\mathcal{M}_\mathcal S} \big(\kappa+ \bigabs{\bfQ} \big)^\frac{p(t,\cdot)-2}{2}\bfQ\, \dx 
\end{align*}
is surjective. Since the function
\begin{align*}
\bigabs{\lambda(\bfQ)} 
=
 \dashint_{\mathcal{M}_\mathcal S} \big(\kappa+ \bigabs{\bfQ} \big)^\frac{p(t,\cdot)-2}{2}\bigabs{\bfQ}\, \dx 
\end{align*}
is continuous and increasing in $\bigabs{\bfQ}$ with $\bigabs{\lambda(\bf0)} =0$ and $\bigabs{\lambda(\bfQ)} \,\rightarrow\,\infty$ as $\bigabs{\bfQ}\,\rightarrow\,\infty$, it follows that $\lambda\,:\, \mathbb{R}^{N\times n}\,\rightarrow\, \mathbb{R}^{N\times n}$ is surjective.
\\With our relation \eqref{eq:Q} in hand, we have that 
\begin{equation}
\begin{aligned}
\label{eq:fQ1}
  \dashint_{\mathcal{M}_\mathcal S} \bigabs{\bfF (t,\cdot,\nabla \bfv) - \bfF (t,\cdot,\bfQ)}^2 &\dx
  \leq \,c\,
    \dashint_{\mathcal{M}_\mathcal S} \bigabs{\bfF (t,\cdot,\nabla \bfv) - \left\langle \bfF (t,\cdot,\nabla \bfv) \right\rangle_{\mathcal{M}_\mathcal{S}}}^2 \dx
    \\&
    +\,c\,
    \dashint_{\mathcal{M}_\mathcal S} \bigabs{\bfF (t,\cdot,\bfQ) - \left\langle \bfF (t,\cdot,\bfQ) \right\rangle_{\mathcal{M}_\mathcal{S}}}^2 \dx.
\end{aligned}
\end{equation}
By Poincar\'{e}'s inequality \eqref{in:crucial0}, we have that
\begin{align}
\label{eq:fQ2}
    \dashint_{\mathcal{M}_\mathcal S} \bigabs{\bfF (t,\cdot,\nabla \bfv) - \left\langle \bfF (t,\cdot,\nabla \bfv) \right\rangle_{\mathcal{M}_\mathcal{S}}}^2 \dx
    \leq\,c\,h^{2\alpha}_\mathcal{S}\,
    \left[\bfF(t,\cdot,\nabla \bfv)\right]^2_{W^{\alpha,2}(\mathcal{M}_\mathcal{S})}.
\end{align}
Now for $x,y\in \mathcal{M}_\mathcal{S}$, the estimates
\begin{align*}
\big|\bfF(t,x,\bfQ)  &- \bfF(t,y,\bfQ)\big|
\\
&\leq \,c\,
\vert p(t,x)-p(t,y)\vert\ln(\kappa + \vert \bfQ \vert)\left((\kappa + \vert \bfQ \vert)^{\frac{p(t,x)-2}{2}} + (\kappa + \vert \bfQ \vert)^{\frac{p(t,y)-2}{2}}\right)\vert\bfQ\vert
\end{align*}
and
\begin{align*}
    \dashint_{\mathcal{M}_\mathcal S} \bigabs{\bfF (t,\cdot,\bfQ) - \left\langle \bfF (t,\cdot,\bfQ) \right\rangle_{\mathcal{M}_\mathcal{S}}}^2 \dx
    \leq \,c\,
        \dashint_{\mathcal{M}_\mathcal S} \dashint_{\mathcal{M}_\mathcal S}  \bigabs{\bfF (t,x,\bfQ) -  \bfF (t,y,\bfQ) }^2 \dx\,\dy
\end{align*}
combined with the H\"{o}lder continuity of $p$ implies that
\begin{align}
\label{eq:fQ3}
    \dashint_{\mathcal{M}_\mathcal S} \bigabs{\bfF (t,\cdot,\bfQ) - \left\langle \bfF (t,\cdot,\bfQ) \right\rangle_{\mathcal{M}_\mathcal{S}}}^2 \dx
    \leq\,c\,h^{2\alpha}_\mathcal{S}    
    \dashint_{\mathcal M_{\mathcal S}}\left[\ln(\kappa+|\bfQ|)\right]^2(\kappa+|\bfQ|)^{p(t,\cdot)}\dx.
\end{align}
Substituting \eqref{eq:fQ3} and \eqref{eq:fQ2} into \eqref{eq:fQ1} yields
\begin{equation}
\begin{aligned}
\label{eq:fQ4}
  \dashint_{\mathcal{M}_\mathcal S} \bigabs{\bfF (t,\cdot,\nabla \bfv) - \bfF (t,\cdot,\bfQ)}&^2 \dx
 \leq\,c\,h^{2\alpha}_\mathcal{S}\,
    \left[\bfF(t,\cdot,\nabla \bfv)\right]^2_{W^{\alpha,2}(\mathcal{M}_\mathcal{S})}
    \\&
    + \,c\,h^{2\alpha}_\mathcal{S}    
    \dashint_{\mathcal M_{\mathcal S}}\left[\ln(\kappa+|\bfQ|)\right]^2(\kappa+|\bfQ|)^{p(t,\cdot)}\dx.
\end{aligned}
\end{equation}
Now, in order to replace $\bfQ$ by $\nabla\bfv$, we first show that the hypothesis of Lemma \ref{lem:pxpy} is satisfied. From the choice of $\bfQ$ in \eqref{eq:Q} we get the estimate
\begin{equation}
\begin{aligned}
\label{eq:fQ4}
\bigabs{\bfQ}^{p^-_{\mathcal{M}_\mathcal{S}}}
&\leq \,c\,\Bigg(\Big|\dashint_{\mathcal{M}_\mathcal{S}}\bfF(t,\cdot,\bfQ)\,\dx \Big|^2+ 1 \Bigg)
=
\,c\,\Bigg(\Big|\dashint_{\mathcal{M}_\mathcal{S}}\bfF(t,\cdot,\nabla\bfv)\,\dx \Big|^2+ 1 \Bigg)
\\
&\leq \,c\,\Bigg(\dashint_{\mathcal{M}_\mathcal{S}}\left|\bfF(t,\cdot,\nabla\bfv)\right|^2\,\dx+ 1 \Bigg) 
\leq \,c\,
\big|  \mathcal{M}_\mathcal{S}\big|^{-1}
\leq \,c\,h^{-n}_\mathcal{S}
\end{aligned}
\end{equation}
using $\bfF(t,\cdot,\nabla \bfv)
  \in L^\infty(0,T;(W^{\alpha,2}(\Omega))^{N\times n})$.
If we define the convex function $\Psi(t):=\left[ \ln(1+t) \right]^2t$, then we can use Lemma \ref{lem:pxpy} for $\ell=n/p^-_{\mathcal{M}_\mathcal{S}}$ to obtain
\begin{equation}
\begin{aligned}
\dashint_{\mathcal M_{\mathcal S}}\left[\ln(\kappa+|\bfQ|)\right]^2(\kappa+|\bfQ|)^{p(t,\cdot)}\dx 
    &\leq
\,c\,    
   \left(\left[\ln(1+|\bfQ|)\right]^2(1+|\bfQ|)^{p^-_{\mathcal{M}_\mathcal{S}}} + 1 \right)
       \\
    &\leq
\,c\,    
   \left(\Psi\left(|\bfQ|^{p^-_{\mathcal{M}_\mathcal{S}}}\right) + 1 \right).
\end{aligned}
\end{equation}
However, we can apply Jensen's inequality to the convex function $\Psi$ to deduce from \eqref{eq:fQ4} that
\begin{equation}
\begin{aligned}
\Psi\left(|\bfQ|^{p^-_{\mathcal{M}_\mathcal{S}}}\right) 
&\leq \,c\,\Bigg(\dashint_{\mathcal{M}_\mathcal{S}}\Psi\left(\left|\bfF(t,\cdot,\nabla\bfv)\right|^2\right)\,\dx+ 1 \Bigg)
\\
&\leq\,c\bigg(\dashint_{\mathcal M_{\mathcal S}}\left[\ln(\kappa+|\nabla\bfv|)\right]^2(\kappa+|\nabla\bfv|)^{p(t,\cdot)}\dx+1\bigg)
\end{aligned}
\end{equation}
which finishes the proof.
\end{proof}

\section{Error analysis}
\label{sec:error}

Let $\{0=t_0<\cdots<t_M=T\}$ be a uniform partition of $[0,T]$ with mesh size $\Delta t=T/M$. Let $\bfu_{0,h}:=\Pi_h\bfu_0$ and for $m\in\{1,\dots,M\}$, we let $\bff_m=(\Delta t)^{-1}\int_{I_m}\bff\ds$, where $I_m=(t_{m-1},t_m)$. Then for every such $m\in\{1,\dots,M\}$, we find  $\bfu_{m,h}\in L^2(\Omega;V_h)$ such that
for all $\bfphi\in V_h$ it holds
\begin{align}\label{tdiscr}
\begin{aligned}
\int_\Omega\bfu_{m,h}(x)\cdot\bfvarphi(x) \dx &+\Delta t\int_\Omega\bfS(t_m,x ,\nabla\bfu_{m,h}(x)):\nabla\bfphi(x)\dx\\&=\int_\Omega\bfu_{m-1,h}(x)\cdot\bfvarphi(x) \dx
+\Delta t\int_\Omega\bff_{m}(x)\cdot\bfvarphi(x) \dx.
\end{aligned}
\end{align}

Our aim now is to establish the convergence rate of the difference between the solution to the continuous problem solving \eqref{eq:heatweak1}, and that of the discrete problem \eqref{tdiscr}.

To do this, we first collect the following assumptions. Throughout the rest of this section we assume that
\begin{align}\label{assumpt} 
 \mathrm{h}\leq c\big(\Delta t\big)^{\frac{1+2\alpha_t}{2\alpha_x}}, \ \bfu_0\in \big(W^{\alpha_x,2}(\Omega)\big)^N, \ \bff=\Div\bfG , \ \bfG \in \big(L^{(p^-)'}(Q)\big)^{N\times n},
\end{align}
where $\alpha_x,\alpha_t\in(0,1]$.
In addition, we suppose that
\begin{align}\label{assumpt2} 
(\Delta t)^{r}\leq\,c\, \inf_{\mathcal S\in\mathscr T_h}{h}_{\mathcal S}
\end{align}
for some $r\geq \frac{1+2\alpha_t}{2\alpha_x}$. Condition \eqref{assumpt2} is needed in the proof of Corollary \ref{cor:bddSimpx} below.
We consider functions $p\in C^{\alpha_x,\alpha_t}(\overline{Q})$ meaning that there is $c\geq0$ such that
\begin{align*}
|p(t,x)-p(s,y)|\leq\,c\big(|t-s|^{\alpha_t}+|x-y|^{\alpha_x}\big)
\end{align*}
for all $s,t\in [0,T]$ and all $x,y\in\overline{\Omega}$.
In the following we state the main result of this section.
\begin{theorem}\label{thm:4}
Let $\bfu$ be the unique weak solution to \eqref{eq:heat} in the sense of Definition \ref{def:weak}, where $p\in C^{\alpha_x,\alpha_t}(\overline{Q})$ with $\alpha_x\in(0,1]$ and $\alpha_t\in(\frac{1}{2},1]$. Moreover, suppose that \eqref{assumpt} and \eqref{assumpt2} hold.
Finally, assume that
\begin{equation}
\begin{aligned}\label{ass:reg}
\bfF(\cdot, \nabla\bfu)\in L^2\big(0,T;\,&\big(W^{\alpha_x,2}(\Omega)\big)^N\big) \,\cap \, W^{\alpha_t,2}(0,T;\big(L^2(\Omega)\big)^N\big),
\\
\bfu&\in L^\infty\big(0,T;\big(W^{\alpha_x,2}(\Omega)\big)^N\big).
\end{aligned}
\end{equation}
Then we have uniformly in $M$ and $\mathrm h$
\begin{align*}
\max_{1\leq m\leq M}\Vert\bfu(t_m,\cdot)-\bfu_{m,h}\Vert_2^2+\Delta t &\sum_{m=1}^M\Vert\bfF(t_m,\cdot,\nabla\bfu(t_m))-\bfF(t_m,\cdot, \nabla\bfu_{m,h})\Vert_2^2
\\
&\leq \,c\,\big( \mathrm{h}^{2\alpha_x}+(\Delta t)^{2\alpha_t}\big),
\end{align*}
where $(\bfu_{m,h})$ is the solution to \eqref{tdiscr}.
\end{theorem}
To proof Theorem \ref{thm:4}, we require the following lemma.

\begin{lemma}
\label{lem:bdd}
Assume that $p\in\PPln(Q)$ with $1<p^-\leq p\leq p^+<\infty$, $\bfu_{0,h}\in \big(L^2(\Omega)\big)^N$ uniformly in $h$ and $\bff=\Div\bfG $ with $\bfG \in L^{(p^-)'}(Q)$. If $\mathbf{u}_{m,h}$ solves \eqref{tdiscr}, then we have
\begin{align*}
\sup_{1\leq m\leq M}\int_\Omega|\bfu_{m,h}(x)|^2\dx
+
\Delta t\sum_{m=1}^M \int_\Omega|\nabla\bfu_{m,h}(x)|^{p(t_m,x)}\dx \,\leq \, c\,(\bfG ,\bfu_0)
\end{align*}
uniformly in $M$ and $\mathrm h$.
\end{lemma}
\begin{proof}

Consider $m=1$ in \eqref{tdiscr} and test with $\mathbf{u}_{1,h}$. Then we obtain by monotonicity and Young's inequality
\begin{align*}
\int_\Omega|\mathbf{u}_{1,h}(x)  |^2\,\mathrm{d}x &+ c'\,\Delta t\int_\Omega\big(|\nabla\mathbf{u}_{1,h}(x)  |^{p(t_1,x)}-1\big)\,\mathrm{d}x \\
&\,\leq \,
\int_\Omega\bfu_{0,h}(x)\cdot \bfu_{1,h}(x) \dx
-\int_{I_1}\int_\Omega\bfG (\sigma,x):\nabla \bfu_{1,h}(x) \dxs
\\
&\leq\frac{1}{2}\int_\Omega|\bfu_{0,h}(x)|^2\dx  +\frac{1}{2}\int_\Omega|\bfu_{1,h}(x)|^2\dx  
\\
&+
c(\delta,p)\int_{I_1}|\bfG (\sigma,x)|^{p'(t_1,x)} \dxs+
\,\delta\Delta t\,\int_\Omega|\nabla\bfu_{1,h}(x)|^{p(t_1,x)} \dx
\end{align*}
for some $c'>0$, where $\delta>0$ is arbitrary.
If we choose $\delta$ small enough we can absorb the last term in the left-hand side. It follows that there exist positive constants $c,C$ depending on $p$ such that
\begin{align*}
\int_\Omega|\mathbf{u}_{1,h}(x)  |^2\,\mathrm{d}x &+ c\,\Delta t\int_\Omega|\nabla\mathbf{u}_{1,h}(x)  |^{p(t_1,x)}\,\mathrm{d}x \\&\leq
\,\int_\Omega|\bfu_{0,h}(x)|^2\dx + C\bigg(\,\int_{I_1}\int_\Omega|\bfG (\sigma,x)|^{(p^-)'} \dxs+1\bigg).
\end{align*}
Iterating this gives
\begin{align*}
\sup_{m}\int_\Omega|\bfu_{m,h}(x)|^2\dx &+\Delta t\sum_{m=1}^M \int_\Omega|\nabla\bfu_{m,h}(x)|^{p(t_m,x)}\dxt\\
 \leq
&\,C'\,\bigg(1+\int_\Omega|\bfu_{0,h}(x)|^2\dx   +
\int_Q|\bfG (\sigma,x)|^{(p^-)'} \dxs \bigg)
\end{align*}
which finishes the proof.
\end{proof}

\noindent The following corollary which is based on  \eqref{assumpt2} seeks to aid us in estimating mixed terms of the form $\left(\bfS(\cdot,\bfeta)-\bfS(\cdot,\bfzeta)\right):( \bfeta- \bfxi)$. It will be used in Lemma \ref{lemma:two} below for the proof of our main theorem, Theorem \ref{thm:4}.
\begin{corollary}
\label{cor:bddSimpx}
Assume that \eqref{assumpt2} holds. Then the estimate
\begin{align*}
\|\nabla \mathbf{u}_{m,h}   \|_{L^\infty(\Omega)}\leq c\left(\Delta t\right)^{-\ell}
\end{align*}
holds  for some $\ell\in\mathbb{N}$.
\end{corollary}
\begin{proof}
We first recall that norms on finite-dimensional spaces are equivalent. Also since each individual summand in Lemma \ref{lem:bdd} is bounded, it follows that
\begin{align*}
\|\nabla \bfu_{m,h}&\|_{L^\infty(\mathcal{S})}  \leq\,c\,   \bigg( \dashint_\mathcal{S}|\nabla \bfu_{m,h}|^{p^\_} \,\dx \bigg)^\frac{1}{p^\_}
\leq\,c\,  \bigg( \dashint_\mathcal{S}\Big( |\nabla \bfu_{m,h}|^{p(t_m,x)} +1\Big)\,\dx \bigg)^\frac{1}{p^\_}
\\
&\leq\,c\,  \bigg( |\mathcal{S}|^{-1} \int_\Omega |\nabla \bfu_{m,h}|^{p(t_m,x)} \,\dx \bigg)^\frac{1}{p^\_} + \,c
\leq\,c\,\left( |\mathcal{S}|^{-1}(\Delta t)^{-1} \right)^\frac{1}{p^\_}
\\
&\leq\,c\,\left( h_\mathcal{S}^{-n}(\Delta t)^{-1} \right)^\frac{1}{p^\_}
\leq\,c\,\left( \Delta t\right)^{-nr-\frac{1}{p^\_}}
\end{align*}
using \eqref{assumpt2}.
We can now take the maximum over all simplexes and the choice of $\ell$ corresponding to the ceiling function of $ nr+\frac{1}{p^\_}$ finishes the proof.
\end{proof}

The following two lemmas will help us to estimate the time discretization error which arises in the elliptic part of the equation.

\begin{lemma}
\label{lem:lemOne}
Let $p\in C^{\alpha_x,\alpha_t}(\overline{Q})$ with $\alpha_x,\alpha_t\in(0,1]$ and $s>1$ be given.
Then for any $\delta>0$ there is $c_{\delta,s}>0$ such that
for all $\sigma,t\in I$ with $|t-\sigma|\leq\Delta t\ll1$ and all $\bfxi,\bfeta\in\R^{N\times n}$
\begin{align*}
\left(\bfS(t,x,\bfxi)-\bfS(\sigma,x,\bfxi)\right):( \bfxi- \bfeta)
&\leq \,
\delta\,  \vert \bfF(t,x, \bfxi)- \bfF(t,x,\bfeta)\vert^2
\\
&+
c_{\delta,s}\,(\Delta t)^{2\alpha_t}\,   \left(1+\left\vert \bfxi\right\vert \right)^{p(t,x)s}
\end{align*}
uniformly in $x\in\Omega$.
\end{lemma}
\begin{proof}
We start with the estimate
\begin{align}
&\vert\bfS(t,x, \bfxi) - \bfS(\sigma,x,\bfxi)\vert\nonumber\\ &\leq c\,\vert p(t,x)-p(\sigma,x)\vert\vert\ln(\kappa + \vert \bfxi \vert)\vert\Big((\kappa + \vert \bfxi \vert)^{p(t,x)-2} + (\kappa + \vert \bfxi \vert)^{p(\sigma,x)-2}\Big)\vert\bfxi\vert\label{smesht}
\\&\leq 
c\, (\Delta t)^{\alpha_t}\big(1+\vert\ln(\kappa+ \vert \bfxi \vert)\vert\big)\left((\kappa + \vert \bfxi \vert)^{p(t,x)-2} + (\kappa + \vert \bfxi \vert)^{p(\sigma,x)-2}\right)\vert\bfxi\vert\nonumber
\end{align}
which follows from $\alpha_t$-H\"{o}lder continuity of $p$. 
Thus we get
\begin{align*}
\left(\bfS(t,x,\bfxi)-\bfS(\sigma,x,\bfxi)\right):(\bfxi- \bfeta) 
&\leq 
c\,\vert  \bfxi- \bfeta\vert (\Delta t)^{\alpha_t} G(t,\sigma,x,\bfxi) (\kappa +\vert \bfxi \vert)^{p(t,x)-2} \vert \bfxi \vert
=:J,
\end{align*}
where
\begin{align*}
G(t,\sigma,x,\bfxi)&:=\big(1+ \vert \ln(\kappa + \vert \bfxi \vert) \vert\big) \left(1 + (\kappa + \vert \bfxi \vert)^{p(\sigma,x)-p(t,x)}\right).
\end{align*}
Note that the claimed inequality is trivial if $\bfxi=0$, so we can exclude this case.
Now for $\sigma\in I\subset\mathbb{R}$ and $x\in\Omega$ fixed, we consider the Young function
\begin{align}\label{eq:2102}
\varphi^{(\sigma,x)}(\eta) :=\int^\eta_0(\kappa +r)^{p(\sigma,x )-2}r\,\mathrm{d}r,\quad t\geq0,
\end{align}
together with its shifted Young function
\begin{align}\label{eq:2102b}
\varphi^{(\sigma,x)}_a(\eta) :=\int^\eta_0\frac{(\varphi^{(\sigma,x)})'(a+\tau)}{a+\tau} \tau\,\mathrm{d}\tau, \quad \eta\geq0,
\end{align}
where $a\geq0$.
Then using this definition, Remark \ref{A9} and Young's inequality (see  Lemma \ref{lem:young}) we find that
\begin{align*}
J&=
c\,\vert  \bfxi- \bfeta\vert   (\Delta t)^{\alpha_t} G(t,\sigma,x,\bfxi) (\varphi^{(t,x)})'( \vert \bfxi \vert)
\\
&\leq
\,\delta\,  \left(\varphi^{(t,x)}_{\vert \bfxi\vert}\right)(\vert \bfxi- \bfeta\vert)
+
c_\delta\,   \left(\varphi^{(t,x)}_{\vert \bfxi \vert}\right)^*\left(G(t,\sigma,x,\bfxi)(\Delta t)^{\alpha_t} (\varphi^{(t,x)})'( \vert \bfxi\vert)\right)
\end{align*}
for any $\delta>0$.
We then use Lemma \ref{lem:hammer} for the first part and Lemmas \ref{lem:shiftedindex} and \ref{lem:shifted2} for the second one and obtain
\begin{align*}
J &\leq
c\,\delta\, \vert \bfF(t,x, \bfxi)- \bfF(t,x,\bfeta)\vert^2
\\
&+
c_\delta\,  \big(G(t,\sigma,x,\bfxi) \big)^{\max\{2,p'(t,x)\}}\left(\varphi^{(t,x)}_{\vert \bfxi\vert}\right)^*\left( (\Delta t)^{\alpha_t}(\varphi^{(t,x)})'( \vert \bfxi\vert)\right)
\\
&\leq
c\,\delta\, \vert \bfF(t,x, \bfxi)- \bfF(t,x,\bfeta)\vert^2
\\
&+
c_\delta(\Delta t)^{2\alpha_t}\,\big(G(t,\sigma,x,\bfxi)\big)^{\max\{2,p'(t,x)\}} \varphi^{(t,x)}( \vert \bfxi\vert).
\end{align*}
Now we claim that
\begin{align}\label{eq:2610}
\big(G(t,\sigma,x,\bfxi)\big)^{\max\{2,p'(t,x)\}} \varphi^{(t,x)}( \vert \bfxi\vert)\leq\,
c\,\left(1+ |\bfxi|\right)^{p(t,x)+2\beta'\max\{2,p'(t,x)\}}  
\end{align}
which finishes the proof. In order to show \eqref{eq:2610} we first note that the left-hand side 
 stays bounded for $|\bfxi|\leq 1$, at least if $\vert t-\sigma \vert \ll 1$. So, we can assume that $|\bfxi|$ is large. In this case we
estimate $\vert \ln(\kappa + \vert \bfxi \vert)\vert\leq \,c\,|\bfxi|^{\beta'}$ for $\beta'>0$ as well as
\begin{align*}
 (\kappa + \vert \bfxi \vert)^{p(\sigma,x)-p(t,x)}
  \leq
  c(1 + \vert \bfxi \vert )^{c(\Delta t)^ {\alpha_t}} 
    \leq c\,
  (1 + \vert \bfxi \vert )^{\beta'} 
\end{align*}
for a small enough choice of $\Delta t$ using the $\alpha_t$-H\"older continuity of $p$. Both together implies \eqref{eq:2610}.
\end{proof}

\begin{lemma}
\label{lemma:two}
Let $p\in \mathcal P(\mathbb R^{n+1})$ with $p_->1$ and $p_+<\infty$.
Then for any $\delta>0$ there is $c_{\delta}>0$ such that
for all $\sigma,t\in I$ with $|t-\sigma|\leq\Delta t\ll1$ and all $\bfxi,\bfeta,\bfzeta\in\R^{N\times n}$ with $|\bfxi|\leq \,C(\Delta t)^{-\ell}$ for some $\ell\in\mathbb N$, we have
\begin{align*}
\left(\bfS(\sigma,x,\bfeta)-\bfS(\sigma,x,\bfzeta)\right):( \bfeta- \bfxi)
&\leq \,
c(\delta,\ell)\Big(|\bfF(\sigma,x,\bfeta)-\bfF(t,x,\bfeta)|^2+ |\bfF(\sigma,x,\bfeta)-\bfF(\sigma,x,\bfzeta)|^2\Big)\\&+\,\delta |\bfF(t,x,\bfeta)-\bfF(t,x,\bfxi)|^2+\delta (\Delta t)^\ell
\end{align*}
uniformly in $x\in\Omega$.
\end{lemma}

\begin{proof}
We start with the estimate
\begin{align*}
I:=\left(\bfS(\sigma,x,\bfeta)-\bfS(\sigma,x,\bfzeta)\right):( \bfeta- \bfxi)\leq\,c \big(\varphi^{(\sigma,x)}_{|\bfeta|}\big)'\big(|\bfeta-\bfzeta|\big)|\bfeta-\bfxi|
\end{align*}
which is a consequence of Lemma \ref{lem:hammer}.
Now we apply Young's inequality, cf. Lemma \ref{lem:young}, for the Young function $\varphi^{(\sigma,x)}_{|\bfeta|}$ such that
\begin{align}
\label{eq:youngg}
I&\leq\,\delta \varphi^{(\sigma,x)}_{|\bfeta|}\big(|\bfeta-\bfxi|\big)
+c(\delta) \big(\varphi^{(\sigma,x)}_{|\bfeta|}\big)^\ast\Big(\big(\varphi^{(\sigma,x)}_{|\bfeta|}\big)'\big(|\bfeta-\bfzeta|\big)\Big),
\end{align}
where $\delta>0$ is arbitrary. We get by Lemmas \ref{lem:hammer} and \ref{lem:shifted2}
\begin{align*}
I&\leq\,c\,\delta (\kappa+|\bfeta|+|\bfxi-\bfeta|)^{p(\sigma,x)-2}|\bfeta-\bfxi|^2
+
c(\delta)\ \varphi^{(\sigma,x)}_{|\bfeta|}\big(|\bfeta-\bfzeta|\big)\\
&\leq\,c\,\delta (\kappa+|\bfeta|+|\bfxi-\bfeta|)^{p(\sigma,x)-2}|\bfeta-\bfxi|^2
+
c(\delta)\ |\bfF(\sigma,x,\bfeta)-\bfF(\sigma,x,\bfzeta)|^2.
\end{align*}
In order to estimate the first term on the right-hand side we distinguish several cases. First, 
we assume that $\kappa+|\bfeta|+|\bfxi-\bfeta|\leq 6C(\Delta t)^{\ell}$ which is only of interest if $\kappa=0$. Under this assumption we have
\begin{align*}
(\kappa+|\bfeta|+|\bfxi-\bfeta|)^{p(\sigma,x)-2}|\bfeta-\bfxi|^2
&\leq (\kappa+|\bfeta|+|\bfxi-\bfeta|)^{p(\sigma,x)}\leq\,c\,(\Delta t)^{\ell p(\sigma,x)}\leq \,c\,(\Delta t)^\ell.
\end{align*}
So, from now on we work under the assumption $\kappa+|\bfeta|+|\bfxi-\bfeta|\geq 6C(\Delta t)^{\ell}$.
Suppose that $|\bfeta|\leq 2C(\Delta t)^{-\ell}$. In this case we have $\kappa+|\bfeta|+|\bfxi-\bfeta|\leq 6C(\Delta t)^{-\ell}$ (since by assumption $|\bfxi| \leq C(\Delta t)^{-\ell}$) and Lemma \ref{lem:pxpy} applies. 
This means that up to a constant, we can replace $p(\sigma,x)$ with $p(t,x)$ and obtain
 \begin{align*}
(\kappa&+|\bfeta|+|\bfxi-\bfeta|)^{p(\sigma,x)-2}|\bfeta-\bfxi|^2
\leq\, c(\ell)\,(\kappa+|\bfeta|+|\bfxi-\bfeta|)^{p(t,x)-2}|\bfeta-\bfxi|^2
\\
&\leq\,c(\ell)\, \Big( \varphi^{(t,x)}_{\vert\bfeta \vert} \Big)' \big(|\bfxi-\bfeta| \big)\,
|\bfeta-\bfxi|
\leq\,c(\ell)\,|\bfF(t,x,\bfeta)-\bfF(t,x,\bfxi)|^2,
\end{align*}
where we have applied Lemma \ref{lem:hammer} in the last step. \\
Now we assume that  $|\bfeta|> 2C(\Delta t)^{-\ell}$. Since by assumption $|\bfxi| \leq C(\Delta t)^{-\ell}$, we have that $|\bfxi| <|\bfeta| $. Furthermore, there is  $ c=c(\kappa)$ such that $\kappa+|\bfeta| \leq c|\bfeta|$ and so
\begin{align*}
(\kappa+|\bfeta|&+|\bfxi-\bfeta|)^{p(\sigma,x)-2}|\bfeta-\bfxi|^2\leq\, c\,|\bfeta|^{p(\sigma,x)}.
\end{align*}
On the other hand, using $|\bfeta|> 2C(\Delta t)^{-\ell}\gg1$,
\begin{align*}
 |\bfeta|^{p(\sigma,x) -2}\leq\,c\,(\kappa + |\bfeta|)^{p(\sigma,x) -2} \Rightarrow |\bfeta|^{p(\sigma,x)}\leq \,c\,(\kappa + |\bfeta|)^{p(\sigma,x) -2} |\bfeta|^2 =c|\bfF(\sigma,x,\bfeta)|^2
\end{align*}
and as such
\begin{align*}
(\kappa+|\bfeta|&+|\bfxi-\bfeta|)^{p(\sigma,x)-2}|\bfeta-\bfxi|^2\leq\, c\,|\bfeta|^{p(\sigma,x)}\leq\,c\,|\bfF(\sigma,x,\bfeta)|^2\\
&\leq\,c\,|\bfF(\sigma,x,\bfeta)-\bfF(t,x,\bfeta)|^2+c\,|\bfF(t,x,\bfeta)|^2.
\end{align*}
Here we have
\begin{align*}
|\bfF(t,x,\bfeta)|^2&\leq\,c\,|\bfeta|^{p(t,x)}\leq\,c\,(\kappa+|\bfeta|+|\bfxi-\bfeta|)^{p(t,x)-2}|\bfeta-\bfxi|^2\\
&\leq\,c\,|\bfF(t,x,\bfeta)-\bfF(t,x,\bfxi)|^2
\end{align*}
by yet another application of Lemma \ref{lem:hammer} in the last inequality above.
By combining the estimates above and replacing $c\delta$ by $\delta$ gives the claim.
\end{proof}
\begin{remark}
In general, the constants in $\sim$ above depends on $p$ so that when one performs the same calculation for $p(t)\mapsto p(t,x)$, one does not get a uniform constant in $x\in\Omega$.
However, this problem is remedied by replacing  these dependence on $p$ by $p^{\pm}$ whichever is appropriate.
\end{remark}

We define the error $\bfe_m:=\bfu(t_m)-\bfu_{m,h}$ and obtain the following estimate.
\begin{lemma}
\label{lem:lemTwo}Let $p\in C^{\alpha_x,\alpha_t}(\overline{Q})$ with $\alpha_x,\alpha_t\in(0,1]$.
There exist a constant $c$ independent of $\Delta t$ and $ \mathrm{h}$ and $s>1$(close to $1$ for $ \mathrm{h}$ small) such that
\begin{align*}
&\Vert \bfe_m\Vert_2^2  +\Delta t\,\Vert\bfF(t_m,\cdot,\nabla\bfu(t_m))-\bfF(t_m,\cdot,\nabla\bfu_{m,h})\Vert_2^2\nonumber\\
&\leq c\,\Bigg[\int_{I_m}\Big(\Vert\bfF(\sigma,\cdot,\nabla\bfu(t_m))-\bfF(\sigma,\cdot,\nabla\bfu(\sigma))\Vert_2^2+ \,\Vert \bfF(t_m,\cdot,\nabla\bfu(t_m))-\bfF(\sigma,\cdot,\nabla\bfu(t_m))\Vert_2^2 \Big)\ds\\
&+  \Delta t\,\Vert \bfF(t_m,\cdot,\nabla\bfu(t_m))-\bfF(t_m,\cdot,\nabla\bfw_m)\Vert_2^2+\,\int_{I_m}\Vert\bfF(\sigma,\cdot,\nabla\bfu(t_m))-\bfF(\sigma,\cdot,\nabla\bfw_m)\Vert_2^2\ds\nonumber
\\&
+
\int_{I_m}\Vert \bfF(t_m,\cdot\nabla\bfu(t_m))-\bfF(\sigma,\cdot,\nabla\bfu(t_m))\Vert^2\ds
+\Vert \bfu(t_m)-\bfw_m \Vert_2^2+\Vert \bfe_{m-1}\Vert_2^2  
\\
&+(\Delta t)^{2\alpha_t+1}\int_\Omega(1+\vert \nabla \bfu(t_m)\vert)^{p(t_m,\cdot)s}\dxs\Bigg]
\nonumber
\end{align*}
for every $\bfw_m\in V_h$ and all $m=1,...,M$.
\end{lemma}
\begin{proof} Subtracting \eqref{tdiscr} from the weak formulation \eqref{eq:heatweak1} we obtain
\begin{align*}
\int_\Omega&\bfe_m\cdot\bfvarphi \dx +\int_{I_m}\int_\Omega\Big(\bfS(\sigma,\cdot,\nabla\bfu(\sigma))-\bfS(t_m,\cdot,\nabla\bfu_{m,h})\Big):\nabla\bfphi\dxs=\int_\Omega\bfe_{m-1}\cdot\bfvarphi \dx
\end{align*}
for every $\bfphi\in V_h$ or equivalently
\begin{align}\label{tdiscrb}
\begin{aligned}
\int_\Omega&\bfe_m\cdot\bfvarphi \dx +\Delta t\int_\Omega\Big(\bfS(t_m,\cdot,\nabla\bfu(t_m))-\bfS(t_m,\cdot,\nabla\bfu_{m,h})\Big):\nabla\bfphi\dx
\\&=\int_\Omega\bfe_{m-1}\cdot\bfvarphi \dx
\dx
+
\int_{I_m}\int_\Omega\Big(\bfS(\sigma,\cdot,\nabla\bfu(t_m))-\bfS(\sigma,\cdot,\nabla\bfu(\sigma))\Big):\nabla\bfphi\dxs
\\&+\int_{I_m}\int_\Omega\Big(\bfS(t_m,\cdot,\nabla\bfu(t_m))-\bfS(\sigma,\cdot,\nabla\bfu(t_m))\Big):\nabla\bfphi\dxs.
\end{aligned}
\end{align}
Setting $\bfphi=\bfw_m-\bfu_{m,h}=\bfe_m-(\bfu(t_m)-\bfw_m)$ we gain
\begin{align*}
&\int_\Omega|\bfe_m|^2 \dx +\Delta t\int_\Omega\Big(\bfS(t_m,\cdot,\nabla\bfu(t_m))-\bfS(t_m,\cdot,\nabla\bfu_{m,h})\Big):\nabla\bfe_m\dx
\\
&=\int_\Omega\bfe_{m-1}\cdot\bfe_m \dx+ \int_\Omega\bfe_{m-1}\cdot(\bfw_m-\bfu(t_m)) \dx
+\int_\Omega\bfe_m\cdot (\bfu(t_m)-\bfw_m)\dx
\\
&+\Delta t\int_\Omega\Big(\bfS(t_m,\cdot,\nabla\bfu(t_m))-\bfS(t_m,\cdot,\nabla\bfu_{m,h})\Big):\nabla(\bfu(t_m)-\bfw_m)\dx.
\\
&+\int_{I_m}\int_\Omega\Big(\bfS(\sigma,\cdot,\nabla\bfu(t_m))-\bfS(\sigma,\cdot,\nabla\bfu(\sigma))\Big):\nabla(\bfw_m-\bfu(t_m))\dxs
\\&+\int_{I_m}\int_\Omega\Big(\bfS(t_m,\cdot,\nabla\bfu(t_m))-\bfS(\sigma,\cdot,\nabla\bfu(t_m))\Big):\nabla(\bfw_m-\bfu(t_m))\dxs
\\& +\int_{I_m}\int_\Omega\Big(\bfS(\sigma,\cdot,\nabla\bfu(t_m))-\bfS(\sigma,\cdot,\nabla\bfu(\sigma))\Big):\nabla\bfe_m\dxs
\\&+\int_{I_m}\int_\Omega\Big(\bfS(t_m,\cdot,\nabla\bfu(t_m))-\bfS(\sigma,\cdot,\nabla\bfu(t_m))\Big):\nabla\bfe_m\dxs
\end{align*}
which can be written as
\begin{align*}
&\int_\Omega\bfe_m\cdot(\bfe_m-\bfe_{m-1}) \dx +\Delta t\int_\Omega\Big(\bfS(t_m,\cdot,\nabla\bfu(t_m))-\bfS(t_m,\cdot,\nabla\bfu_{m,h})\Big):\nabla\bfe_m\dx
\\&=\int_\Omega(\bfe_m - \bfe_{m-1})\cdot (\bfu(t_m)-\bfw_m)\dx
\\
&+\Delta t\int_\Omega\Big(\bfS(t_m,\cdot,\nabla\bfu(t_m))-\bfS(t_m,\cdot,\nabla\bfu_{m,h})\Big):\nabla(\bfu(t_m)-\bfw_m)\dx
\\
&+\int_{I_m}\int_\Omega\Big(\bfS(\sigma,\cdot,\nabla\bfu(t_m))-\bfS(\sigma,\cdot,\nabla\bfu(\sigma))\Big):\nabla(\bfw_m-\bfu(t_m))\dxs
\\&+\int_{I_m}\int_\Omega\Big(\bfS(t_m,\cdot,\nabla\bfu(t_m))-\bfS(\sigma,\cdot,\nabla\bfu(t_m))\Big):\nabla(\bfw_m-\bfu(t_m))\dxs
\\& +\int_{I_m}\int_\Omega\Big(\bfS(\sigma,\cdot,\nabla\bfu(t_m))-\bfS(\sigma,\cdot,\nabla\bfu(\sigma))\Big):\nabla\bfe_m\dxs
\\&+\int_{I_m}\int_\Omega\Big(\bfS(t_m,\cdot,\nabla\bfu(t_m))-\bfS(\sigma,\cdot,\nabla\bfu(t_m))\Big):\nabla\bfe_m\dxs.
\end{align*}
Now, we apply the identity $\bfa\cdot(\bfa-\bfb)=\frac{1}{2}\big(|\bfa|^2-|\bfb|^2+|\bfa-\bfb|^2\big)$ (which holds for any $\bfa,\bfb\in\mathbb R^N$) with $\bfa=\bfe_m$ and $\bfb=\bfe_{m-1}$) to the first term on the right-hand side to get
\begin{align*}
&\frac{1}{2}\left(\Vert\bfe_m\Vert_2^2- \Vert \bfe_{m-1}\Vert_2^2+\Vert\bfe_m-\bfe_{m-1}\Vert_2^2\right)  +c\,\Delta t\Vert\bfF(t_m,\cdot,\nabla\bfu(t_m))-\bfF(t_m,\cdot,\nabla\bfu_{m,h})\Vert_2^2
\\
&\leq\int_\Omega(\bfe_m -\bfe_{m-1})\cdot (\bfu(t_m)-\bfw_m)\dx
\\
&+\Delta t\int_\Omega\Big(\bfS(t_m,\cdot,\nabla\bfu(t_m))-\bfS(t_m,\cdot,\nabla\bfu_{m,h})\Big):\nabla(\bfu(t_m)-\bfw_m)\dx
\\
&+\int_{I_m}\int_\Omega\Big(\bfS(\sigma,\cdot,\nabla\bfu(t_m))-\bfS(\sigma,\cdot,\nabla\bfu(\sigma))\Big):\nabla(\bfw_m-\bfu(t_m))\dxs
\\&+\int_{I_m}\int_\Omega\Big(\bfS(t_m,\cdot,\nabla\bfu(t_m))-\bfS(\sigma,\cdot,\nabla\bfu(t_m))\Big):\nabla(\bfw_m-\bfu(t_m))\dxs
\\&+\int_{I_m}\int_\Omega\Big(\bfS(\sigma,\cdot,\nabla\bfu(t_m))-\bfS(\sigma,\cdot,\nabla\bfu(\sigma))\Big):\nabla\bfe_m\dxs
\\&+\int_{I_m}\int_\Omega\Big(\bfS(t_m,\cdot,\nabla\bfu(t_m))-\bfS(\sigma,\cdot,\nabla\bfu(t_m))\Big):\nabla\bfe_m\dxs
\\
&=I_1+\cdots+I_{6}\end{align*}
using also Lemma \ref{lem:hammer}.
Now Young's inequality yields
\begin{align*}
I_1&\leq\,\frac{1}{2}\,\Vert\bfe_m-\bfe_{m-1}\Vert_2^2 + \frac{1}{2}\,\Vert\bfu(t_m)-\bfw_m\Vert_2^2.
\end{align*}
By Lemma \ref{lem:young2}, we obtain
\begin{align*}
I_2
&\leq\,\delta\,\Delta t\, \Vert \bfF(t_m,\cdot,\nabla\bfu(t_m))-\bfF(t_m,\cdot,\nabla\bfu_{m,h}))\Vert_2^2\\&+c_\delta\Delta t\,\Vert\bfF(t_m,\cdot,\nabla\bfu(t_m))-\bfF(t_m,\cdot,\nabla\bfw_m)\Vert_2^2,
\\
I_3
&\leq\,\delta \,\int_{I_m}\Vert \bfF(\sigma,\cdot,\nabla\bfu(t_m))-\bfF(\sigma,\cdot,\nabla\bfw_m)\Vert_2^2\ds\\& +c_\delta\,\int_{I_m}\Vert \bfF(\sigma,\cdot,\nabla\bfu(t_m))-\bfF(\sigma,\cdot,\nabla\bfu(\sigma))\Vert_2^2\ds.
\end{align*}
Now we use Lemma \ref{lemma:two} which applies due to Corollary \ref{cor:bddSimpx}. We obtain
\begin{align*}
I_5
&\leq\,\delta\Delta t\int_\Omega |\bfF(t_m,\cdot\nabla\bfu(t_m))-\bfF(t_m,\cdot,\nabla\bfu_{m,h})|^2\dx+\delta(\Delta t)^{2\alpha_t+1}\\
&+c(\delta)\int_{I_m}\int_\Omega |\bfF(t_m,\cdot\nabla\bfu(t_m))-\bfF(\sigma,\cdot,\nabla\bfu(t_m))|^2\dxs\\
&+c(\delta)\int_{I_m}\int_\Omega |\bfF(\sigma,\cdot,\nabla\bfu(t_m))-\bfF(\sigma,\cdot,\nabla\bfu(\sigma))|^2\dxs
\end{align*}
for every $\delta>0$. 
Now, using Lemma \ref{lem:lemOne} we also get that
\begin{align*}
I_4 +I_6 
&\leq
\,\delta\Delta t\,\Vert \bfF(t_m,\cdot,\nabla\bfu(t_m))-\bfF(t_m,\cdot,\nabla\bfw_m)\Vert_2^2 \\&+\delta
\Delta t\,\Vert \bfF(t_m,\cdot,\nabla\bfu(t_m))-\bfF(t_m,\cdot,\nabla\bfu_{m,h})\Vert_2^2 
\\
&+c_\delta(\Delta t)^{2\alpha_t+1}\int_\Omega\left(1+| \nabla \bfu(t_m)| \right)^{p(t_m,\cdot)s}\dx
\end{align*}
for every $\delta>0$.
Collecting the above results with an appropriate choice of $\delta$ yields the claim.
\end{proof}

Based on Lemma \ref{lem:lemTwo} and the regularity assumptions in \eqref{ass:reg} 
we are now able to finish the proof of Theorem \ref{thm:4}. Note that so far we have not used the assumption $\alpha_t>\frac{1}{2}$. But it will be crucial in the following.

\begin{proof}[Proof of Theorem \ref{thm:4}]
We denote the terms appearing on the right-hand side of the inequality in Lemma \ref{lem:lemTwo} by $J_1,...,J_8$ and estimate them one by one.
First of all, we use the following estimate
\begin{align}
\vert\bfF(t,x, \bfxi) - \bfF(\sigma,x,\bfxi)\vert &\leq c\,\vert p(t,x)-p(\sigma,x)\vert\vert\ln(\kappa + \vert \bfxi \vert)\vert\left((\kappa + \vert \bfxi \vert)^{\frac{p(t,x)-2}{2}} + (\kappa + \vert \bfxi \vert)^{\frac{p(\sigma,x)-2}{2}}\right)\vert\bfxi\vert\nonumber
\\&\leq 
c\, (\Delta t)^{\alpha_t}\vert\ln(\kappa + \vert \bfxi \vert)\vert\left((\kappa + \vert \bfxi \vert)^{\frac{p(t,x)-2}{2}} + (\kappa + \vert \bfxi \vert)^{\frac{p(\sigma,x)-2}{2}}\right)\vert\bfxi\vert\nonumber
\\
&\leq c\, (\Delta t)^{\alpha_t}\vert\ln(\kappa + \vert \bfxi \vert)\vert\left(1+(\kappa + \vert \bfxi \vert)^{c(\Delta t)^{\alpha_t}}\right) (\kappa + \vert \bfxi \vert)^{\frac{p(t,x)-2}{2}}\vert\bfxi\vert\nonumber
\\
&\leq 
c_s\, (\Delta t)^{\alpha_t}(1 + \vert \bfxi \vert)^{\frac{p(t,x)s}{2}}
\label{eq:2808}
\end{align}
 which holds for arbitrary $s>1$
provided $\Delta t$ is small enough, cf. the proof of Lemma \ref{lem:lemOne}. So we get 
\begin{align*}
J_1
&
\leq \int_{I_m}\Vert\bfF(t_m,\cdot,\nabla\bfu(t_m))-\bfF(\sigma,\cdot,\nabla\bfu(\sigma))\Vert_2^2 \ds\\
&+c_s\,(\Delta t)^{2\alpha_t+1}\int_\Omega(1+\vert \nabla \bfu(t_m)\vert)^{p(t_m,\cdot)s}\dx
\end{align*}
for any $s>1$. Now we use \eqref{ass:reg} together with Corollary \ref{cor:crucial}
to obtain
\begin{align}\label{eq:2308}\begin{aligned}
\int_{I_m}&\Vert\bfF(t_m,\cdot,\nabla\bfu(t_m))-\bfF(\sigma,\cdot,\nabla\bfu(\sigma))\Vert_2^2 \ds\\&\leq \,c(\Delta t)^{2\alpha_t}\left[\bfF(\cdot,\cdot,\nabla\bfu)\right]^2_{W^{\alpha_t,2}(I_m;L^2(\Omega))}.
\end{aligned}
\end{align}
Note that this estimate requires $\alpha_t>\frac{1}{2}$. $J_2$ can be estimated in the same fashion. Furthermore,  we can estimate $J_5$ as in \eqref{eq:2308}.
By the choice of $\bfw_m=\Pi_h \bfu(t_m)$, we obtain from Lemma \ref{thm:app_V} (by choosing $\bfQ=\bfQ(\sigma)$ appropriately)
\begin{align*}
J_3&\leq \,c\Delta t\,\Vert \bfF(t_m,\cdot,\nabla\bfu(t_m))-\bfF(t_m,\cdot,\bfQ)\Vert_2^2+c\,\mathrm{h}^2\\
&\leq 
\,c\int_{I_m}\,\Vert \bfF(\sigma,\cdot,\nabla\bfu(\sigma))-\bfF(\sigma,\cdot,\bfQ)\Vert_2^2\ds+\,c\int_{I_m}\,\Vert \bfF(t_m,\cdot,\bfQ)-\bfF(\sigma,\cdot,\bfQ)\Vert_2^2\ds\\
&+\,c\int_{I_m}\,\Vert \bfF(t_m,\cdot,\nabla\bfu(t_m))-\bfF(\sigma,\cdot,\nabla\bfu(\sigma))\Vert_2^2\ds+c\mathrm{h}^2\\
&=: J_3^1+J_3^2+J_3^3+c\,\mathrm{h}^2.
\end{align*}
On account of \eqref{eq:app_V2} combined with a similar argument as in \eqref{eq:2808} (using instead $\alpha_x$-H\"{o}lder continuity), we obtain
\begin{align*}
J_3^1&\leq\,c\, \mathrm{h}^{2\alpha_x}\,
    \int_{I_m}\left[\bfF(\sigma,\cdot,\nabla \bfu(\sigma))\right]^2_{W^{\alpha_x,2}(\Omega)}\ds
    \\&  
+c\,\mathrm{h}^{2\alpha_x}\int_{I_m}\int_{\Omega}\ln(\kappa+|\nabla\bfu(\sigma)|)^2(\kappa+|\nabla\bfu(\sigma)|)^{p(\sigma,\cdot)}\dxs+c\,\mathrm{h}^{2\alpha_x}
\\
&\leq\,c\, \mathrm{h}^{2\alpha_x}\,
    \int_{I_m}\left[\bfF(\sigma,\cdot,\nabla \bfu(\sigma))\right]^2_{W^{\alpha_x,2}(\Omega)}\ds   
 \\&
  +c\, \mathrm{h}^{2\alpha_x}\int_{I_m}\int_{\Omega}(1+|\nabla\bfu(\sigma)|)^{p(\sigma,\cdot)s}\dxs +c\,\mathrm{h}^{2\alpha_x}.
\end{align*}
By \eqref{eq:2308} we have
\begin{align*}
J_3^3&\leq \,c(\Delta t)^{2\alpha_t}\left[\bfF(\cdot,\cdot,\nabla\bfu)\right]^2_{W^{\alpha_t,2}(I_m;L^2(\Omega))}.
\end{align*}
In order to estimate $J_3^2$ we use that $\bfF(\cdot,\cdot,\nabla\bfu)\in L^\infty(0,T;\big(L^2(\Omega)\big)^{N\times n})$ which follows from \eqref{ass:reg}, $\alpha_t>\frac{1}{2}$ and Lemma \ref{thm:sobboch}.
Hence, \eqref{eq:2808b} yields
\begin{align*}
J_3^2&\leq\,c(\Delta t)^{2\alpha_t}\int_{I_m}\int_\Omega\ln(\kappa+\vert \bfQ\vert)^2(\kappa+\vert \bfQ\vert)^{p(\sigma,\cdot)}\dxs\\
&+\,c(\Delta t)^{2\alpha_t}\int_{I_m}\int_\Omega\ln(\kappa+\vert \bfQ\vert)^2(\kappa+\vert \bfQ\vert)^{p(t_m,\cdot)}\dxs\\
&\leq\,c(\Delta t)^{2\alpha_t}\int_{I_m}\int_\Omega\big(\ln(\kappa+\vert \nabla \bfu(\sigma)\vert)^2(\kappa+\vert \nabla \bfu(\sigma)\vert)^{p(\sigma,\cdot)}+1\big)\dxs
\\
&+\,c(\Delta t)^{2\alpha_t}\int_{I_m}\int_\Omega\big(\ln(\kappa+\vert \nabla \bfu(\sigma)\vert)^2(\kappa+\vert \nabla \bfu(\sigma)\vert)^{p(t_m,\cdot)}+1\big)\dxs
\\
&
\leq\,c(\Delta t)^{2\alpha_t}\int_{I_m}\int_\Omega(1+\vert \nabla \bfu(\sigma)\vert)^{p(\sigma,\cdot)s}\dxs\\
&
+\,c(\Delta t)^{2\alpha_t}\int_{I_m}\int_\Omega(1+\vert \nabla \bfu(\sigma)\vert)^{p(t_m,\cdot)s}\dxs
\end{align*}
using also \eqref{eq:2808}. In order to estimate $J_4$ we use \eqref{eq:app_V1} to get
\begin{align*}
J_4&\leq \,c\,\int_{I_m}\Vert \bfF(\sigma,\cdot,\nabla\bfu(t_m))-\bfF(\sigma,\cdot,\bfQ)\Vert_2^2\ds+c\,\mathrm{h}^2
\\&\leq
\,c\int_{I_m}\,\Vert \bfF(\sigma,\cdot,\nabla\bfu(\sigma))-\bfF(\sigma,\cdot,\bfQ)\Vert_2^2\ds\\
&+
\,c\int_{I_m}\,\Vert \bfF(t_m,\cdot,\nabla\bfu(t_m))-\bfF(\sigma,\cdot,\nabla\bfu(\sigma))\Vert_2^2\ds\\
&+
\,c\int_{I_m}\,\Vert \bfF(t_m,\cdot,\nabla\bfu(t_m))-\bfF(\sigma,\cdot,\nabla\bfu(t_m))\Vert_2^2\ds
+c\,\mathrm{h}^2
\\&= J_4^1+J_4^2+J_4^3+c\,\mathrm{h}^2.
\end{align*}
Note that $J_4^1$ coincides with $J_3^1$, whereas $J_4^2$ coincides with $J_3^3$. Now on account of \eqref{eq:2808}, we obtain
\begin{align*}
J_4^3\leq
&c\,(\Delta t)^{2\alpha_t+1}\int_\Omega(1+\vert \nabla \bfu(t_m)\vert)^{p(t_m,\cdot)s}\dx.
\end{align*}
Finally, we have
\begin{align*}
J_6\leq \,c\, \mathrm{h}^{2\alpha_x}\left[\bfu(t_m)\right]^2_{W^{\alpha_x,2}(\Omega)}
\end{align*}
using \eqref{eq:stab''}.
Plugging all together shows
\begin{align*}
&\Vert \bfe_m\Vert_2^2  +\Delta t\,\Vert\bfF(t_m,\cdot,\nabla\bfu(t_m))-\bfF(t_m,\cdot,\nabla\bfu_{m,h})\Vert_2^2\\
&\leq \,c(\Delta t)^{2\alpha_t}\left[\bfF(\cdot,\cdot,\nabla\bfu)\right]^2_{W^{\alpha_t,2}(I_m;L^2(\Omega))}
+  c \,\mathrm{h}^{2\alpha_x}\,
   \left[\bfF(\sigma,\cdot,\nabla \bfu(\sigma))\right]^2_{L^2\left(I_m;W^{\alpha_x,2}(\Omega)\right)}
\\&
+c\,\mathrm{h}^{2\alpha_x} \,\left[\bfu(t_m)\right]^2_{W^{\alpha_x,2}(\Omega)} +c\big((\Delta t)^{2\alpha_t}+\mathrm{h}^{2\alpha_x}\big)\int_{I_m}\int_\Omega(1+\vert \nabla \bfu(\sigma)\vert)^{p(\sigma,\cdot)s}\dxs
\\&
\,c(\Delta t)^{2\alpha_t+1}\int_\Omega(1+\vert \nabla \bfu(t_m)\vert)^{p(t_m,\cdot)s}\dx+\,c(\Delta t)^{2\alpha_t}\int_{I_m}\int_\Omega(1+\vert \nabla \bfu(\sigma)\vert)^{p(t_m,\cdot)s}\dxs\\
&+c\,\mathrm{h}^{2\alpha_x} +\Vert \bfe_{m-1}\Vert_2^2.
\end{align*}
Iterating this and applying Gronwall's lemma implies (increasing $s$ as the case may be)
\begin{align*}
&\sup_{0\leq m\leq M}\Vert \bfe_m\Vert_2^2  +\Delta t\sum_{m=1}^M\,\Vert\bfF(t_m,\cdot,\nabla\bfu(t_m))-\bfF(t_m,\cdot,\nabla\bfu_{m,h})\Vert_2^2\\
&\leq 
\,c(\Delta t)^{2\alpha_t}\left[\bfF(\cdot,\cdot,\nabla\bfu)\right]^2_{W^{\alpha_t,2}(0,T;L^2(\Omega))}
+  c\,\mathrm{h}^{2\alpha_x}\,
    \left[\bfF(\sigma,\cdot,\nabla \bfu(\sigma))\right]^2_{L^2\left(0,T;W^{\alpha_x,2}(\Omega)\right)}
\\&
+c\frac{\mathrm{h}^{2\alpha_x}}{\Delta t}\sup_{t\in(0,T)}\left[\bfu(t)\right]^2_{W^{\alpha_x,2}(\Omega)}
 +c\big((\Delta t)^{2\alpha_t}+\mathrm{h}^{2\alpha_x}\big)\sup_{t\in(0,T)}\int_\Omega(1+\vert \nabla \bfu(t)\vert)^{p(t,\cdot)s}\dx 
+ \|\bfe_0\|_2^2.
\end{align*}
We can  then apply \eqref{eq:stab''} to $\bfe_0$ (recalling that $\bfu_{0,h} :=\Pi_h\bfu_0$) which will subsequently be bounded as a result of \eqref{assumpt}. This shows the claim as all integrals on the right-hand side are finite by \eqref{ass:reg}. The only term in  need of an explanation is
$$\sup_{t\in(0,T)}\int_\Omega(1+\vert \nabla \bfu(t)\vert)^{p(t,\cdot)s}\dx\leq \,c\bigg(\sup_{t\in(0,T)}\int_\Omega |\bfF(t,\cdot,\nabla\bfu(t))|^{2s}\dx+1\bigg).$$
Now, due to Lemmas \ref{lem:inter} (with $\theta$ such that $\theta\alpha_t>\frac{1}{2}$) and \ref{thm:sobboch} we have $$\bfF(\cdot,\cdot,\nabla\bfu)\in C^0([0,T];L^{2s}(\Omega))$$ for some $s>1$ (close enough to 1) due to the Sobolev embedding $W^{\theta\alpha_x,2}(\Omega)\hookrightarrow L^{\frac{2n}{n-2\theta\alpha_x}}(\Omega)$, cf. Lemma \ref{lema}. In particular $\int_\Omega |\bfF(t,\cdot,\nabla\bfu(t))|^{2s}\dx$ is bounded in time. The proof of Theorem \ref{thm:4} is hereby complete noticing that $\mathrm{h}^{2\alpha_x}\leq \,c(\Delta t)^{1+2\alpha_t}$ by \eqref{assumpt}.
\end{proof}

\section{Regularity of solutions}
\label{sec:reg}

Regularity of $\bfu$ is best studied via the nonlinear tensor
$$\bfF(t,x,\bfxi)=(\kappa+|\bfxi|)^{\frac{p(t,x)-2}{2}}\bfxi.$$
We consider exponents $p\in C^{\alpha_x,\alpha_t}(\overline{Q})$, where $\alpha_x,\alpha_t\in(0,1]$.
It is expected that H\"older-continuity of $p$ transfers to fractional differentiability of
$\bfF(\cdot,\cdot,\nabla\bfu)$ with the same exponent. Results in a similar fashion can be found, in particular, in \cite{AcMiSe}. 
\begin{theorem}
\label{thm:quasispace}
Let  $\alpha_x,\alpha_t \in (0,1]$ be given and let $\Omega$ be a bounded $C^{1,\alpha_x}$-domain. Assume that $p\in C^{\alpha_x,\alpha_t}(\overline{Q})$ with $p^->1$ and $p^+<\infty$. 
Let $\bfu$ be the unique weak solution to \eqref{eq:heat} in the sense of Definition \ref{def:weak} with 
\begin{align}\label{ass:fu0}
\begin{aligned}
&\bff\in \big(\mathcal V^{p'(\cdot)}_0(Q)\big)^N\cap C^{\alpha_t}\big(0,T;\big(L^2(\Omega)\big)^N\big),\\ &\bfu_0\in \big(W^{1,2}_0(\Omega)\big)^N,\quad \Div\bfS(0,\cdot,\nabla\bfu_0)\in \big(L^2(\Omega)\big)^N.
\end{aligned}
\end{align}
Then we have
\begin{align}
\label{quasispace}
\begin{aligned}
\bfF(\cdot,\cdot,\nabla\bfu)&\in L^2\big(0,T;\big(W^{\tilde\alpha_x,2}(\Omega)\big)^{N\times n}\big)\cap W^{\tilde\alpha_t,2}\big(0,T;\big(L^2(\Omega)\big)^{N\times n}\big),\\
\bfu&\in L^\infty(0,T;\big(W^{\tilde\alpha_x,2}(\Omega)\big)^N),
\end{aligned}
\end{align}
for all $\tilde\alpha_x<\alpha_x$ and all $\tilde\alpha_t<\alpha_t$.
\end{theorem}
\begin{remark}
If $\alpha_x=1$ ($\alpha_t=1$) then \eqref{quasispace} holds for $\tilde\alpha_x=1$ ($\tilde\alpha_t=1$).
\end{remark}
\emph{Proof of Theorem \ref{thm:quasispace}.}
\vspace{-1ex}
\subsection{Preparations} In order to guarantee that all terms in the following computations are well-defined, we have to regularize the problem. For $\delta>0$, we consider the problem
\begin{align}
\left\{\begin{array}{rc}
\partial_t \bfu_\delta=\Div\bfS_\delta(\nabla\bfu_\delta)+\bff& \mbox{in $Q=(0,T)\times\Omega $,}\\
\bfu_\delta=0\qquad\quad& \mbox{ \,on $(0,T)\times\partial\Omega$,}\\
\bfu_\delta(0,\cdot)=\bfu_{0}\qquad& \mbox{in $\Omega$,}\end{array}\right.\label{eq:heatdelta}
\end{align}
where the nonlinear tensor $\bfS_\delta$ is given by
\begin{align}\label{eq:stressdelta}
\bfS_\delta(\nabla \bfu_\delta)=\bfS_\delta(t,x,\nabla \bfu_\delta)=(\kappa+|\nabla\bfu_\delta|)^{p(t,x)-2}\nabla\bfu_\delta+\delta(\kappa+|\nabla\bfu_\delta|)^{q-2}\nabla\bfu_\delta
\end{align}
with $q>\max\{2,p^{+}\}$. For $\delta>0$ fixed, this is a problem with standard growth conditions. The existence of a unique weak solution
$$\bfu_\delta\in C([0;T];\big(L^2(\Omega)\big)^N)\cap L^q(0,T;\big(W^{1,q}_0(\Omega)\big)^N)$$
to \eqref{eq:heatdelta} can be shown by monotone operator theory. Moreover, we have the uniform estimate
\begin{align*}
\sup_{t\in(0,T)} \int_\Omega |\bfu_\delta|^2\dx+\int_0^T\int_\Omega|\nabla\bfu_\delta|^{\px}\dxt+\delta\int_0^T\int_\Omega|\nabla\bfu_\delta|^{q}\dxt\leq\,c(\bfu_0,\bff).
\end{align*}
This follows formally be testing \eqref{eq:heatdelta} with $\bfu_\delta$. Consequently, we have
\begin{align*}
\bfu_\delta&\rightharpoonup \bfu\quad\text{in}\quad L^{p^{-}}(0,T;\big(W_0^{1,p^{-}}(\Omega)\big)^N),\\
\bfu_\delta&\rightharpoonup^\ast \bfu\quad\text{in}\quad L^\infty(0,T;\big(L^2(\Omega)\big)^N),\\
\nabla\bfu_\delta&\rightharpoonup\nabla\bfu\quad\text{in}\quad \big(L^{\px}(Q)\big)^{N\times n},\\
\delta^{1/q}\nabla\bfu_\delta&\rightarrow0\quad\text{in}\quad \big(L^{q}(Q)\big)^{N\times n},
\end{align*}
where $\bfu$ is the unique weak solution to \eqref{eq:heat}.
For ease of presentation, we neglect the $\delta$-regularization layer in the following. We leave the necessary changes (where all estimates below have to be shown uniformly in $\delta$) to the reader.\\

\subsection{Local regularity in space}
Let us fix a ball 
$B\Subset \Omega$ with radius $r_B$ and $10B\Subset\Omega$. The aim is to show that
\begin{align}\label{eq:1902a}
\begin{aligned}
\sup_{t\in (0,T)}[\bfu]^2_{W^{\tilde\alpha_x,2}(B)}&+\int_0^T[\bfF(\cdot,\cdot,\nabla\bfu)]^2_{W^{\tilde\alpha_x,2}(B)}\dt
\\&
\leq\,c_B\int_0^T\int_{10B} \big(1+|\nabla \bfu|^{sp}\big)\dxt
\end{aligned}
\end{align}
for all $\tilde\alpha_x<\alpha_x$ and some $s>1$ (close to 1).
We let $e_\gamma$ be the unit vector in the direction $x_\gamma$  for $\gamma=1,...,n$ and define $\tau^\gamma_h v(x)=v(x+he_\gamma)-v(x)$ for some $h\in\mathbb{R}/\{0\}$. We further set $\Delta^\gamma_h v(x)=(\tau^\gamma_h v(x))/h$. 
We  assume that $|h|\leq h_0<r_B$ and consider a cut-off function $\xi\in C^\infty_0(3B)$. We also assume that $\xi=1$ in $2B$, $0\leq\xi\leq1$ and $|\nabla\xi|\leq c/r_B$ for some constant $c>0$. Equation \eqref{eq:heatweak1} implies
\begin{align*}
\int_\Omega\tau^\gamma_h\bfu\cdot \bfvarphi\dx&
=
\int_\Omega\tau^\gamma_h\bfu_0\cdot \bfvarphi\dx
-
\int_0^t\int_\Omega \tau^\gamma_h\big\{\bfS(\cdot,\cdot,\nabla\bfu)\big\}:\nabla\bfvarphi\dxs
 \\
 &+
\int_0^t\int_\Omega \tau^\gamma_h\bff\cdot\bfvarphi\dxs
\end{align*}
for all $\bfphi\in C^\infty_0(3B)$. We use a spatial mollifier $(\cdot)_\varrho$ with parameter $\varrho\ll1$ and gain
\begin{align*}
\int_\Omega\big(\tau^\gamma_h\bfu\big)_\varrho\cdot \bfvarphi\dx&
=
\int_\Omega\big(\tau^\gamma_h\bfu_0\big)_\varrho\cdot \bfvarphi\dx
-
 \int_0^t\int_\Omega\Big( \tau^\gamma_h\big\{\bfS(\cdot,\cdot,\nabla\bfu)\big\}\Big)_\varrho:\nabla\bfvarphi\dxs
 \\
 &+
 \int_0^t\int_\Omega \big(\tau^\gamma_h\bff\big)_\varrho\cdot\bfvarphi\dxs
\end{align*}
for all $\bfphi\in C^\infty_0(3B)$. Since all involved terms are smooth, the fundamental lemma in the calculus of variations yields
\begin{align*}
\big(\tau^\gamma_h\bfu\big)_\varrho(t)&
=
\big(\tau^\gamma_h\bfu_0\big)_\varrho
+
 \int_0^t\Div\Big( \tau^\gamma_h\big\{\bfS(\cdot,\cdot,\nabla\bfu)\big\}\Big)_\varrho\ds
+
 \int_0^t \big(\tau^\gamma_h\bff\big)_\varrho\ds
\end{align*}
in $3B$. As the right-hand side is differentiable in time, so is the left-hand side and we obtain
\begin{align*}
\partial_t\big(\tau^\gamma_h\bfu\big)_\varrho&
=
\Div\Big( \tau^\gamma_h\big\{\bfS(\cdot,\cdot,\nabla\bfu)\big\}\Big)_\varrho
+\big(\tau^\gamma_h\bff\big)_\varrho,\quad \big(\tau^\gamma_h\bfu\big)_\varrho(0)=\big(\tau^\gamma_h\bfu_0\big)_\varrho.
\end{align*}
Multiplying by $\xi^q\big(\tau^\gamma_h\bfu\big)_\varrho$, where $q>1$ will be fixed later, and integrating over $(0,t)\times 3B$ shows
\begin{align*}
\frac{1}{2}\int_\Omega\xi^q&\big|\big(\tau^\gamma_h\bfu\big)_\varrho\big|^2\dx
+
 \int_0^t\int_\Omega\xi^q\Big( \tau^\gamma_h\big\{\bfS(\cdot,\cdot,\nabla\bfu)\big\}\Big)_\varrho:\big(\tau^\gamma_h\nabla\bfu\big)_\varrho \dxs
 \\
 &=
\frac{1}{2}\int_\Omega\xi^q\big|\big(\tau^\gamma_h\bfu_0\big)_\varrho\big|^2\dx
-q
\int_0^t\int_\Omega\xi^{q-1}\Big( \tau^\gamma_h\big\{\bfS(\cdot,\cdot,\nabla\bfu)\big\}\Big)_\varrho:\nabla\xi\otimes \big(\tau^\gamma_h\bfu\big)_\varrho \dxs
 \\
 &+
 \int_0^t\int_\Omega\xi^q \big(\tau^\gamma_h\bff\big)_\varrho\cdot\big(\tau^\gamma_h\bfu\big)_\varrho \dxs.
\end{align*}
By standard properties of the convolution we can pass to the limit $\varrho\rightarrow0$ and obtain
\begin{align}\label{eq:difweakeqn}
\begin{aligned}
\frac{1}{2}\int_\Omega\xi^q&\big|\tau^\gamma_h\bfu \big|^2\dx
+
\int_0^t\int_\Omega \xi^q \tau^\gamma_h\big\{\bfS(\cdot,\cdot,\nabla\bfu)\big\}:\tau^\gamma_h\nabla\bfu \dxs
 \\
 &=
\frac{1}{2}\int_\Omega\xi^q\big|\tau^\gamma_h\bfu_0 \big|^2\dx
-q
\int_0^t\int_\Omega\xi^{q-1} \,\tau^\gamma_h\big\{\bfS(\cdot,\cdot,\nabla\bfu)\big\}:\nabla\xi\otimes \tau^\gamma_h\bfu \dxs
 \\
 &+
\int_0^t\int_\Omega \xi^q \tau^\gamma_h\bff \cdot\tau^\gamma_h\bfu \dxs.
\end{aligned}
\end{align}
Let us consider the term
\begin{align*}
(O)
&:=
\int_0^t\int_\Omega \xi^q \tau_h^\gamma\Big\{\bfS(\sigma,x,\nabla\bfu)\Big\}:\tau^\gamma_h\nabla\bfu\dxs
\\&=
\int_0^t\int_\Omega \xi^q \Big\{\bfS(\sigma ,x+he_\gamma,\nabla\bfu(\sigma,x+he_\gamma))-\bfS(\sigma,x,\nabla\bfu(\sigma,x+he_\gamma))\Big\}:\tau^\gamma_h\nabla\bfu\dxs
\\&
+\int_0^t\int_\Omega \xi^q \Big\{\bfS(\sigma,x,\nabla\bfu(\sigma,x+he_\gamma))-\bfS(\sigma,x,\nabla\bfu(\sigma,x))\Big\}:\tau^\gamma_h\nabla\bfu\dxs
\\&
=:(O)_1+(O)_2.
\end{align*}
Notice that by choosing $h$ small enough, we have
\begin{align}\label{eq:hxS}
\begin{aligned}
|&\bfS(t,x+h e_\gamma,\bfxi)-\bfS(t,x,\bfxi)|\\&\leq c\,\vert p(x+h e_\gamma)-p(x)\vert\vert\ln (\kappa+\vert\bfxi\vert)\vert\Big\{ (\kappa+\vert\bfxi\vert)^{p(x+h e_\gamma)-2} +(\kappa+\vert\bfxi\vert)^{p(x)-2} \Big\}\vert \bfxi\vert
\\
&\leq c\,h^{\alpha_x} (\kappa+\vert\bfxi\vert)^{\frac{p(x)-2}{2}}\vert\ln (\kappa+\vert\bfxi\vert)\vert\Big\{ (\kappa+\vert\bfxi\vert)^{\frac{p(x)-2}{2}+p(x+h e_\gamma)-p(x)} +(\kappa+\vert\bfxi\vert)^{\frac{p(x)-2}{2}} \Big\}\vert \bfxi\vert
\\
&\leq c\,h^{\alpha_x} (\kappa+\vert\bfxi\vert)^{\frac{p(x)-1}{2}} (1+\vert\bfxi\vert)^{\frac{sp(x)-1}{2}} 
\end{aligned}
\end{align}
by the H\"older-continuity of $p$.  Here, $s>1$ can be chosen arbitrarily close to one if $h$ is small enough. Similarly, there holds
\begin{align}\label{eq:hxF}
|&\bfF(t,x+h e_\gamma,\bfxi)-\bfF(t,x,\bfxi)|
\leq c\,h^{\alpha_x} (1+\vert\bfxi\vert)^{\frac{sp(x)}{2}}.
\end{align}
We use the abbreviation $f^\theta$ for a $\theta$-shift of a function $f:\Omega\rightarrow\R$ in direction $e_\gamma$, that is $f^\theta(x)=f(x+\theta e_\gamma)$ whenever it is well-defined.
Now by using the last estimate in \eqref{eq:hxS} (note that the right-hand side is increasing in $\bfxi$) as well as Young's inequality, we gain
\begin{align}
|(O)_1|
&
\leq\,c\,h^{\alpha_x} \int_0^t\int_{\Omega} \xi^q\big(\kappa+|\nabla\bfu^h|\big)^{\frac{p-1}{2}}\big(1+|\nabla\bfu^h|\big)^{\frac{sp-1}{2}}|\tau_h^\gamma\nabla\bfu|\dxs\nonumber\\
&
\leq\,c\,h^{\alpha_x} \int_0^t\int_{\Omega} \xi^q\big(\kappa+|\nabla\bfu|+|\nabla\bfu^h|\big)^{\frac{p-1}{2}}\big(1+|\nabla\bfu|+|\nabla\bfu^h|\big)^{\frac{sp-1}{2}}|\tau_h^\gamma\nabla\bfu|\dxs\nonumber\\
&
\leq\,c\,h^{\alpha_x} \int_0^t\int_{\Omega} \xi^q\big(\kappa+|\nabla\bfu|+|\nabla\bfu^h|\big)^{\frac{p-2}{2}}\big(1+|\nabla\bfu|+|\nabla\bfu^h|\big)^{\frac{sp}{2}}|\tau_h^\gamma\nabla\bfu|\dxs\nonumber\\
&
\leq\,\varepsilon \int_0^t\int_{\Omega} \xi^q\big(\kappa+|\nabla\bfu|+|\nabla\bfu^h|\big)^{p-2}|\tau_h^\gamma\nabla\bfu|^2\dxs\nonumber
\\&
+c(\varepsilon) \,h^{2\alpha_x}\int_0^t\int_{\Omega}
\xi^q\big(1+|\nabla\bfu|+|\nabla\bfu^h|\big)^{sp}\dxs,\label{O1}
\end{align}
where $\varepsilon>0$ is arbitrary and $p=p(\sigma,x)$ in the above as well as in subsequent computations below. 
Moreover, the following holds
\begin{align*}
(O)_2&\sim\, \int_0^t\int_{\Omega}\xi^q(\kappa+|\nabla \bfu|+|\nabla\bfu^h|)^{p-2}|\tau_h^\gamma\nabla\bfu|^2\dxs
\\
&
\geq\,\int_0^t\int_{\Omega}\xi^q|\tau_h^\gamma\bfF(\cdot,\cdot,\nabla\bfu)|^2\dxs
-
ch^{2\alpha_x}\int_0^t\int_{3B}(1+|\nabla\bfu^h|^{sp})\dxs\\
&
\geq\,\int_0^t\int_{\Omega}\xi^q|\tau_h^\gamma\bfF(\cdot,\cdot,\nabla\bfu)|^2\dxs
-
ch^{2\alpha_x}\int_0^t\int_{4B}(1+|\nabla\bfu|^{sp})\dxs
\end{align*}
using Lemma \ref{lem:hammer} as well as the estimate \eqref{eq:hxF}.
 The last term in \eqref{eq:difweakeqn} is easily controlled. We have that
\begin{align*}
\int_0^t\int_{\Omega}\xi^q\tau^\gamma_h\bff\cdot\tau^\gamma_h\bfu \dxs
&\leq h^2\,\left(\int_0^t\int_{3B} \vert\Delta^\gamma_h\bff\vert^{p'}\dxs\right)^\frac{1}{p'} \left(\int_0^t\int_{3B}\vert \Delta^\gamma_h\bfu\vert^{p}
\dxs\right)^\frac{1}{p}\leq\,ch^2
\end{align*}
recalling that $\bff\in \big(\mathcal V^{p'(\cdot)}(Q)\big)^N$ and $\bfu\in (\mathcal{V}^{p(\cdot)}(Q))^N$.
 It remains to control the second term on the right-hand side of \eqref{eq:difweakeqn} which we represent by $(I)$. We have
\begin{align*}
(I)&=-q\int_0^t\int_{\Omega} \xi^{q-1}\Big\{\bfS(\sigma,x+he_\gamma,\nabla\bfu(\sigma,x+he_\gamma))-\bfS(\sigma,x,\nabla\bfu(\sigma,x+he_\gamma))\Big\}:\nabla\xi\otimes\tau_h^\gamma\bfu\dxs
\\&
-q\int_0^t\int_{\Omega}\xi^{q-1}\Big\{\bfS(\sigma,x,\nabla\bfu(\sigma,x+he_\gamma))-\bfS(\sigma,x,\nabla\bfu(\sigma ,x))\Big\}:\nabla\xi\otimes\tau_h^\gamma\bfu\dxs
\\&
=(I)_1+(I)_2.
\end{align*}
Using \eqref{eq:hxS} we obtain
\begin{align*}
(I)_1&\leq\,c_B h^{\alpha_x+1}\int_0^t\int_{3B} \big(1+|\nabla\bfu^h|\big)^{s(p-1)}|\Delta_h^\gamma\bfu|\dxs
\\
&
\leq\,c_Bh^{\alpha_x+1}\int_0^t\int_{3B} \big(1+|\nabla\bfu^h|\big)^{sp}\dxs
+c_Bh^{\alpha_x+1}\int_0^t\int_{3B}|\Delta_h^\gamma\bfu|^{p}\dxs,
\end{align*}
where we applied Young's inequality pointwise with $p$ and $p'=p/(p-1)$. Note that $s>1$ can be chosen arbitrarily close to 1.
For the term involving the difference quotient 
we use $\Delta_h^\gamma\bfu(\sigma,x)=\int_0^1\partial_\gamma\bfu(\sigma,x+\theta he_\gamma)\,\dd\theta$ and Jensen's inequality to obtain
\begin{align}\label{eq:0203}
\begin{aligned}
\int_0^t\int_{3B}|\Delta_h^\gamma\bfu|^p\dxs&\leq\,c\,\int_0^1\int_0^t\int_{3B} \big(1+ |\nabla \bfu^{\theta h}|\big)^{p(\sigma,x)}\dxs\,\dd\theta\\
&\leq\,c\,\int_0^1\int_0^t\int_{3B} \big(1+ |\nabla \bfu^{\theta h}|\big)^{sp(\sigma,x+\theta h e_\gamma)}\dxs\,\dd\theta\\
&\leq\,c\,\int_0^t\int_{4B}\big(1+|\nabla\bfu|\big)^{ps}\dxs.
\end{aligned}
\end{align}
In order to estimate $(I)_2$ we use ideas from \cite[Lemma 12]{DieE08} based upon shifted Young functions (see appendix \ref{sec:Orlicz spaces}).
 Using the function $\varphi^{(\sigma,x)}$ and its shifted version $\varphi_a^{(\sigma,x)}$  introduced in \eqref{eq:2102} and \eqref{eq:2102b} we have that
\begin{align*}
(I)_2&\leq \,c\,h\,\int_0^t\int_{\Omega}\xi^{q-1}\Big(\varphi_{|\nabla\bfu|}^{(\sigma,x)}\Big)'(|\tau_h^\gamma\nabla\bfu|)|\nabla\xi||\Delta_h^\gamma\bfu|\dxs\\
&\leq \,c\,\dashint_0^h\int_0^t\int_{\Omega}\xi^{q-1}\Big(\varphi_{|\nabla\bfu|}^{(\sigma,x)}\Big)'(|\tau_h^\gamma\nabla\bfu|)|\nabla\xi|h|\nabla\bfu^\theta|\dxs\,\dd\theta\\
&\leq \,c\frac{h}{h_0}\,\dashint_0^h\int_0^t\int_{\Omega}\xi^{q-1}\Big(\varphi_{|\nabla\bfu^\theta|}^{(\sigma,x)}\Big)'(|\nabla\bfu^h-\nabla\bfu^\theta|)|\nabla\xi|h_0|\nabla\bfu^\theta|\dxs\,\dd\theta\\
&+ \,c\,\frac{h}{h_0}\dashint_0^h\int_0^t\int_{\Omega}\xi^{q-1}\Big(\varphi_{|\nabla\bfu^\theta|}^{(\sigma,x)}\Big)'(|\nabla\bfu-\nabla\bfu^\theta|)|\nabla\xi|h_0|\nabla\bfu^\theta|\dxs\,\dd\theta
\end{align*}
as a consequence of the last estimate from Lemma \ref{lem:hammer} and Lemma \ref{lem:shift_ch} (b). 
Using Young's inequality from Lemma \ref{lem:young} for
$\varphi^{(\sigma,x)}_{|\nabla\bfu^\theta|}$
we obtain further
\begin{align*}
(I)_2
&\leq \,\varepsilon\,\frac{h}{h_0}\dashint_0^h\int_0^t\int_{\Omega}\Big(\varphi_{|\nabla\bfu^\theta|}^{(\sigma,x)}\Big)^\ast\Big(\xi^{q-1}\Big(\varphi_{|\nabla\bfu^\theta|}^{(\sigma,x)}\Big)'(|\nabla\bfu^h-\nabla\bfu^\theta|)\Big)\dxs\,\dd\theta\\
&+ \,\varepsilon\,\frac{h}{h_0}\dashint_0^h\int_0^t\int_{\Omega}\Big(\varphi_{|\nabla\bfu^\theta|}^{(\sigma,x)}\Big)^\ast\Big(\xi^{q-1}\Big(\varphi_{|\nabla\bfu^\theta|}^{(\sigma,x)}\Big)'(|\nabla\bfu-\nabla\bfu^\theta|)\Big)\dxs\,\dd\theta\\
&+ \,c(\varepsilon)\,\frac{h}{h_0}\dashint_0^h\int_0^t\int_{\Omega}\Big(\varphi_{|\nabla\bfu^\theta|}^{(\sigma,x)}\Big)\Big(h_0|\nabla\xi||\nabla\bfu^\theta|\Big)\dxs\,\dd\theta,
\end{align*}
where $\varepsilon>0$ is arbitrary.
Choosing $q$ large enough such that (recall that $\xi\leq1$)
$$\max \{(\xi^{q-1})^{p'(\sigma,x)},(\xi^{q-1})^{2}\}\leq \max \{(\xi^{q-1})^{(p^+)'},(\xi^{q-1})^{2}\}\leq \xi^{(q-1)\min\{(p^+)',2\}}\leq \xi^q$$
we infer from Lemmas \ref{lem:shiftedindex} and \ref{lem:shifted2}
\begin{align*}
(I)_2
&\leq \,\varepsilon\,\frac{h}{h_0}\dashint_0^h\int_0^t\int_{\Omega}\xi^q\Big(\varphi_{|\nabla\bfu^\theta|}^{(\sigma,x)}\Big)^\ast\Big(\Big(\varphi_{|\nabla\bfu^\theta|}^{(\sigma,x)}\Big)'(|\nabla\bfu^h-\nabla\bfu^\theta|)\Big)\dxs\,\dd\theta\\
&+ \,\varepsilon\,\frac{h}{h_0}\dashint_0^h\int_0^t\int_{\Omega}\xi^q\Big(\varphi_{|\nabla\bfu^\theta|}^{(\sigma,x)}\Big)^\ast\Big(\Big(\varphi_{|\nabla\bfu^\theta|}^{(\sigma,x)}\Big)'(|\nabla\bfu-\nabla\bfu^\theta|)\Big)\dxs\,\dd\theta\\
&+ \,c(\varepsilon,B)\,\frac{h}{h_0}\dashint_0^h\int_0^t\int_{3B}\Big(\varphi_{|\nabla\bfu^\theta|}^{(\sigma,x)}\Big)\Big(h_0|\nabla\bfu^\theta|\Big)\dxs\,\dd\theta\\
&\leq \,\varepsilon\,\frac{h}{h_0}\dashint_0^h\int_0^t\int_{\Omega}\xi^q\varphi_{|\nabla\bfu^\theta|}^{(\sigma,x)}(|\nabla\bfu^h-\nabla\bfu^\theta|)\dxs\,\dd\theta\\
&+ \,\varepsilon\,\frac{h}{h_0}\dashint_0^h\int_0^t\int_{\Omega}\xi^q\varphi_{|\nabla\bfu^\theta|}^{(\sigma,x)}(|\nabla\bfu-\nabla\bfu^\theta|)\dxs\,\dd\theta\\
&+ \,c(\varepsilon,B)h_0^2\,\dashint_0^h\int_0^t\int_{3B}\varphi^{(\sigma,x)}\big(|\nabla\bfu^\theta|\big)\dxs\,\dd\theta\\
&=:\varepsilon\frac{h}{h_0}(I)_2^1+\varepsilon\frac{h}{h_0}(I)_2^2+c(\varepsilon,B)h_0^2(I)_2^3.
\end{align*}
As in \eqref{eq:0203} we obtain
\begin{align*}
(I)_2^3&\leq\,c\,\dashint_0^h\int_0^t\int_{3B} \big(1+ |\nabla \bfu^{\theta}|\big)^{p(\sigma,x)}\dxs\,\dd\theta\\
&\leq\,c\,\dashint_0^h\int_0^t\int_{3B} \big(1+ |\nabla \bfu^{\theta}|\big)^{sp(\sigma,x+\theta e_\gamma)}\dxs\,\dd\theta\\
&\leq\,c\,\int_0^t\int_{4B}\big(1+|\nabla\bfu|^{ps}\big)\dxs.
\end{align*}
Furthermore, there holds by Lemma \ref{lem:hammer} and \eqref{eq:hxF}
\begin{align*}
(I)_2^2&\leq \,c\dashint_0^h\int_0^t\int_{\Omega}\xi^q|\bfF(\cdot,\cdot,\nabla\bfu)-\bfF(\cdot,\cdot,\nabla\bfu^\theta)|^2\dxs\,\dd\theta\\
&\leq \,c\dashint_0^h\int_0^t\int_{\Omega}\xi^q|\tau_\theta^\gamma\bfF(\cdot,\cdot,\nabla\bfu)|^2\dxs\,\dd\theta+c\,h_0^{2\alpha_x}\,\int_0^t\int_{4B}\big(1+|\nabla\bfu|^{ps}\big)\dxs
\end{align*}
as well as
\begin{align*}
(I)_2^1&\leq \,c\dashint_0^h\int_0^t\int_{\Omega}\xi^q|\bfF(\cdot,\cdot,\nabla\bfu^h)-\bfF(\cdot,\cdot,\nabla\bfu^\theta)|^2\dxs\,\dd\theta\\
&\leq \,c\int_0^t\int_{\Omega}\xi^q|\bfF(\cdot,\cdot,\nabla\bfu^h)-\bfF(\cdot,\cdot,\nabla\bfu)|^2\dxs\\&+\,c\dashint_0^h\int_0^t\int_{\Omega}\xi^q|\bfF(\cdot,\cdot,\nabla\bfu)-\bfF(\cdot,\cdot,\nabla\bfu^\theta)|^2\dxs\,\dd\theta\\
&\leq \,c\int_0^t\int_{\Omega}\xi^q|\bfF(\cdot,\cdot,\nabla\bfu^h)-\bfF(\cdot,\cdot,\nabla\bfu)|^2\dxs+\,c\dashint_0^h\int_0^t\int_{3B}|\tau_{\theta}^\gamma\bfF(\cdot,\cdot,\nabla\bfu)|^2\dxs\,\dd\theta\\&+c\,h_0^{2\alpha_x}\,\int_0^t\int_{4B}\big(1+|\nabla\bfu|^{ps}\big)\dxs.
\end{align*}
Note that by Lemma \ref{lem:hammer} the first term on the right-hand side is proportional to $(\mathcal O)_2$ and can consequently be absorbed.
Plugging all together and choosing $\varepsilon$ small enough,  we conclude
\begin{align}\label{eq:1902aaa}
\begin{aligned}
&\sup_{t\in [0,T]}\int_{2B}|\tau^\gamma_h\bfu|^2\dx+\int_0^T\int_{2B}|\tau_h^\gamma\bfF(\cdot,\cdot,\nabla\bfu)|^2\dxt
\\&
\,\,\leq\frac{1}{2}\frac{h}{h_0}\dashint_0^h\int_0^T\int_{3B}|\tau_{\theta}^\gamma\bfF(\cdot,\cdot,\nabla\bfu)|^2\dxt\,\dd\theta+\,c_B\,h_0^{2\alpha_x}\int_0^T\int_{4B} \big(1+ |\nabla \bfu|^{sp}\big)\dxt
\end{aligned}
\end{align}
as well as
\begin{align}\label{eq:1902aab}
\int_0^T\int_{2B}|\tau_h^\gamma\bfF(\cdot,\cdot,\nabla\bfu)|^2\dxt
&
\leq\frac{1}{2}\frac{h}{h_0}\dashint_0^h\int_0^T\int_{4B}|\tau_{\theta}^\gamma\bfF(\cdot,\cdot,\nabla\bfu)|^2\dxt\,\dd\theta\\&+\,c_B\,h_0^{2\alpha_x}\int_0^T\int_{4B} \big(1+ |\nabla \bfu|^{sp}\big)\dxt,\nonumber\\
\label{eq:1902aac}
\dashint_0^{h_0}\int_0^T\int_{2B}|\tau_h^\gamma\bfF(\cdot,\cdot,\nabla\bfu)|^2\dxt\,\mathrm{d}h
&
\leq\frac{1}{2}\dashint_0^{h_0}\int_0^T\int_{4B}|\tau_{h}^\gamma\bfF(\cdot,\cdot,\nabla\bfu)|^2\dxt\,\dd h\\&+\,c_B\,h_0^{2\alpha_x}\int_0^T\int_{4B} \big(1+ |\nabla \bfu|^{sp}\big)\dxt.\nonumber
\end{align}
Using \cite[Lemma 13]{DieE08} we can get rid of the $\frac{1}{2}$-term on the right-hand side of \eqref{eq:1902aac} as well as the additional $h$-integral (by slightly reducing the radius) to gain
\begin{align*}
\int_0^T\int_{\frac{8}{5}B}|\tau_h^\gamma\bfF(\cdot,\cdot,\nabla\bfu)|^2\dxt
&
\leq\,c_B\,h_0^{2\alpha_x}\int_0^T\int_{4B} \big(1+ |\nabla \bfu|^{sp}\big)\dxt.
\end{align*}
Using this inequality in \eqref{eq:1902aac} we finally obtain
\begin{align}\label{eq:1902aa}
\begin{aligned}
\sup_{t\in [0,T]}\int_{2B}&|\tau^\gamma_h\bfu|^2\dx+\int_0^T\int_{2B}|\tau_h^\gamma\bfF(\cdot,\cdot,\nabla\bfu)|^2\dxt
\\&
\leq\,c_B\,h^{2\alpha_x}\int_0^T\int_{10B} \big(1+ |\nabla \bfu|^{sp}\big)\dxt.
\end{aligned}
\end{align}
The final estimate \eqref{eq:1902a} now follows  from Lemma \ref{le2} by using $\bfF(\cdot,\cdot,\nabla\bfu)\in \big(L^2(Q)\big)^{N\times n}$.\\\

\subsection{Boundary regularity in space}
We will show an estimate in the spirit of \eqref{eq:1902a} for boundary points.
This will be done by 
flattening the boundary (introducing local coordinates) and 
reflecting the solution on the boundary (which requires zero boundary data). This method was originally introduced in \cite{ChenDiBe}. We follow the approach from
\cite{BCDS}, to which we refer for further details. The approach from
\cite{BCDS} is in turn inspired by
\cite{BuDiSc}.\\
Without loss of generality, we assume that $0$ is a boundary point and the outer normal of $\Omega$ at $0$ is $(0',-1)^t$. To avoid confusion, we have denoted the zero vectors in $\setR^n$ and $\setR^{n-1}$ by $0$ and $0'$ respectively.
Now, for $R>0$ small enough and $B_{R}(0') \subset \R^{n-1}$, there is a
local coordinate map $\psi:B_{R}(0')\to \setR$ of class $C^{1,\alpha_x}$ such that 
\begin{align*}
\partial\Omega\cap B_R(0)&\subset\set{(x',\psi(x'))\,|\,x'\in B_{R}(0')},\\  \Omega\cap B_R(0)&\subset \set{(x',x_n)\in B_{R}(0)\,:\,x_n>\psi(x')},
\end{align*}
where we have used the abbreviation $(x_1,\ldots,x_{n-1}, x_n):= (x',x_n)$. Without loss of generality we assume that  $\psi(0)=0$ and $\nabla\psi(0)=0$.
 We further define 
\begin{align*}
\Psi:\overline{\Omega}\cap B_{R}(0) &\to  \set{(y',y_n)\in \setR^n\,:\,y_n\geq 0}
\\
\Psi(x',x_n)&=(x',x_n-\psi(x')).
\end{align*}
We notice that $\Psi(x)\in \set{(y',y_n):y_n=0}$ for $x\in\partial\Omega$.  Since we assume that $\Omega$ has a Lipschitz boundary, $\Psi$ has a well defined inverse $\Psi^{-1}$ which is  Lipschitz continuous provided $R$ is small enough. 
Finally, we set
$$\bfJ(x) = \nabla \Psi (x) \quad \hbox{for $x \in \overline{\Omega}\cap B_{R}(0)$.}$$\\
%
%
%
%
Clearly, we have $\det(\bfJ)\equiv 1$ and $\det(\bfJ^{-1})\equiv 1$. This means that the mappings $\Psi$ and $\Psi^{-1}$ are volume preserving.
 Moreover, due to the Lipschitz property of the boundary, we may assume that there are constants $0<\lambda\leq 1\leq \Lambda$ such that
\begin{equation}\label{lL}
\lambda\leq |\bfJ(x)|\leq \Lambda
\end{equation}
uniformly in $x$ provided $R$ is small enough.
This implies, in particular, that
 \begin{align}\label{inclusion1}
 & B_{\lambda r}^+(0)\subset \Psi(\overline\Omega\cap B_r(0)) \subset B_{\Lambda r}^+(0),
 \\ \label{inclusion2}
 &  \overline\Omega\cap B_{\frac r\Lambda }(0)\subset \Psi^{-1}(B_r^+(0)) \subset \overline\Omega\cap B_{\frac{r}{\lambda}}(0),
 \end{align}
for $0\leq r\leq R$.
Here we denoted for a given $\varrho>0$
\begin{align*}
B^+_\varrho(0) &:= B_\varrho(0) \cap \{ (x',x_n) \,:\, x_n \geq 0\},\\
B^-_\varrho(0) &:= B_\varrho(0) \cap \{ (x',x_n) \,:\, x_n \leq 0\}.
\end{align*}
We now set $y=\Psi(x)$ and for a given function $\bfg:\overline{\Omega}\cap B_R(0)\to \setR^N$ we define $\tilde{\bfg}: \Psi (\overline{\Omega}\cap B_R(0)) \to \overline{\Omega}\cap B_R(0)$ by
$$\tilde{\bfg}(y)=\bfg (\Psi^{-1}(y)) \quad \hbox{for $y \in \Psi (\overline{\Omega}\cap B_R(0))$}.$$
Hence, we obtain
\[
\nabla \tilde{\bfg}(y)=\nabla\bfg(\Psi^{-1}(y))\bfJ^{-1}(\Psi^{-1}(y)) \quad \hbox{for $y \in \Psi (\overline{\Omega}\cap B_R(0))$}
\]
and
\[
 \nabla\bfg(x)=\nabla\tilde{\bfg}(\Psi(x))\bfJ(x) \quad \hbox{for $x \in \overline{\Omega}\cap B_R(0)$} 
\]
for any differentiable function $\bfg$. 
Applying this change of variable to the solution $\bfu$ of \eqref{eq:heat}
we obtain the system (to be understood in the weak sense, cf. Definition \ref{def:weak})
\begin{align}
\label{eq:boundaryflat}
\partial_t \tilde\bfu=\Div\tilde\bfS(\nabla\tilde\bfu)+\tilde\bff
\end{align}
 in $(0,T)\times B_{\lambda R}^+(0)$ with initial datum $\tilde\bfu_0$, where 
\begin{align*}
\tilde\bfS(t,x,\bfeta)&=\bfS(t,\Psi^{-1}(x),\bfeta\bfJ(\Psi^{-1}(x)))\bfJ^t(\Psi^{-1}(x)),\\
\tilde\bff(t,x)&=\bff(t,\Psi^{-1}(x)),\quad  \tilde\bfu_0(x)=\bfu_0(\Psi^{-1}(x)),
\end{align*}
recall \eqref{inclusion1}. We claim that the operator $\tilde\bfS$ satisfies 
\eqref{growth1} and \eqref{growth2} (with $p$ replaced by $\tilde p(t,x)=p(t,\Psi^{-1}(x))$ as long as $\bfS$ does. In fact, we have
\begin{align*}
D_\bfeta\tilde\bfS(t,x,\bfeta)(\bfxi,\bfxi)&=D_\bfeta\bfS(t,\Psi^{-1}(x),\bfeta\bfJ(\Psi^{-1}(x)))(\bfJ(\Psi^{-1}(x))\bfxi,\bfJ(\Psi^{-1}(x))\bfxi)\\&\sim (\kappa+|\bfeta\bfJ(\Psi^{-1}(x))|)^{\tilde p-2}|\bfJ(\Psi^{-1}(x))\bfxi|^2\\
&\sim (\kappa+|\bfeta|)^{\tilde p-2}|\bfxi|^2
\end{align*}
using \eqref{growth1} as well as \eqref{lL}.
Moreover, there holds
\begin{align*}
&\tilde\bfS(t,x,\bfeta)-\tilde\bfS(t,y,\bfeta)\\&=\bfS(t,\Psi^{-1}(x),\bfeta\bfJ(\Psi^{-1}(x)))\bfJ^t(\Psi^{-1}(x))-\bfS(t,\Psi^{-1}(y),\bfeta\bfJ(\Psi^{-1}(y)))\bfJ^t(\Psi^{-1}(y))\\
&=\Big[\bfS(t,\Psi^{-1}(x),\bfeta\bfJ(\Psi^{-1}(x)))-\bfS(t,\Psi^{-1}(y),\bfeta\bfJ(\Psi^{-1}(x)))\Big]\bfJ^t(\Psi^{-1}(x))\\
&+\Big[\bfS(t,\Psi^{-1}(y),\bfeta\bfJ(\Psi^{-1}(x)))-\bfS(t,\Psi^{-1}(y),\bfeta\bfJ(\Psi^{-1}(y)))\Big]\bfJ^t(\Psi^{-1}(x))\\
&+\bfS(t,\Psi^{-1}(y),\bfeta\bfJ(\Psi^{-1}(y)))\Big[\bfJ^t(\Psi^{-1}(x))-\bfJ^t(\Psi^{-1}(y))\Big]\\
&=(I)+(II)+(III).
\end{align*}
We obtain, using \eqref{growth2} for $\bfS$ as well as Lipschitz continuity of $\Psi^{-1}$ and \eqref{lL}, that
\begin{align*}
(I)&\leq\,c\,|\Psi^{-1}(x)-\Psi^{-1}(y)|^{\alpha_x}\big(1+|\ln(\kappa+|\bfeta\bfJ(\Psi^{-1}(x))|)|\big) 
\\&\times\Big( (\kappa+|\bfeta\bfJ(\Psi^{-1}(x))|)^{\tilde p(t,x)-2} + (\kappa+|\bfeta\bfJ(\Psi^{-1}(x))|)^{\tilde p(t,y)-2} \Big) |\bfeta\bfJ(\Psi^{-1}(x))|.
\end{align*}
In the case $\kappa+|\bfeta\bfJ(\Psi^{-1}(x))|\leq 1$, we argue by the elementary inequalities
\begin{align*}
|\ln(\kappa+|\bfeta\bfJ(\Psi^{-1}(x))|)|&=-\ln(\kappa+|\bfeta\bfJ(\Psi^{-1}(x))|)\leq -\ln(\lambda(\kappa+|\bfeta|))\\&=-\ln \lambda-\ln(\kappa+|\bfeta|)\leq c\,(1+ |\ln(\kappa+|\bfeta|)|),
\end{align*}
which follow from \eqref{lL}.
Otherwise, we have
\begin{align*}
|\ln(\kappa+|\bfeta\bfJ(\Psi^{-1}(x))|)|&=\ln(\kappa+|\bfeta\bfJ(\Psi^{-1}(x))|)\leq \ln(\Lambda(\kappa+|\bfeta|))\\&\leq\,c(1+ |\ln(\kappa+|\bfeta|)|),
\end{align*}
which is again a consequence of \eqref{lL}.
Combining both cases yields
\begin{align*}
(I)
&\leq\,c\,|x-y|^{\alpha_x}\big(1+|\ln(\kappa+|\bfeta|)|\big) \Big( (\kappa+|\bfeta|)^{\tilde p(t,x)-2} + (\kappa+|\bfeta|)^{\tilde p(t,y)-2} \Big) |\bfeta|.
\end{align*}
using again Lipschitz continuity of $\Psi^{-1}$ as well as  \eqref{lL}.
By Lemma \ref{lem:hammer} we obtain
\begin{align*}
(II)&\leq \,c\,\big(\kappa+|\bfeta(\bfJ(\Psi^{-1}(x))-\bfJ(\Psi^{-1}(y)))|+|\bfeta\bfJ(\Psi^{-1}(x))|\big)^{\tilde p(t,y)-2}\\&\qquad\qquad\qquad\times|\bfeta(\bfJ(\Psi^{-1}(x))-\bfJ(\Psi^{-1}(y)))|\\&\leq\,c\,|x-y|^{\alpha_x}\big(\kappa+|\bfeta|\big)^{\tilde p(t,y)-2}|\bfeta|
\end{align*}
by H\"older-continuity of $\bfJ\circ\Psi^{-1}$.
Finally, we have
\begin{align*}
(III)\leq\,c\,|x-y|^{\alpha_x}\big(\kappa+|\bfeta\bfJ(\Psi^{-1}(x))|\big)^{\tilde p-2}|\bfeta\bfJ(\Psi^{-1}(x))|\leq\,c\,|x-y|^{\alpha_x}\big(\kappa+|\bfeta|\big)^{\tilde p-2}|\bfeta|
\end{align*}
using again H\"older-continuity of $\bfJ\circ\Psi^{-1}$ and \eqref{lL}.
Combining the estimates above shows
that 
\begin{align}\label{eq:1610}
\text{the operator }\tilde\bfS \text{ satisfies 
\eqref{growth1} and \eqref{growth2} }
\end{align}
with $p$ replaced by $\tilde p(t,x)=p(t,\Psi^{-1}(x))$.\\
Now we reflect the problem at the flat boundary.
We introduce the reflection matrix $\bfR\in \setR^{N\times n}$, $\bfR:=\text{diag}(1,...,1,-1)$.
Now we extend the solution by setting
\begin{align*}
\overline\bfu(t,x)&:=\left\{\begin{aligned}
&\tilde\bfu(t,x) &&\textrm{for } (x',x_n)\in B_{\lambda R}(0)\,:\,x_n\geq 0,\\
&-\tilde\bfu(t,x',-x_n) &&\textrm{for }(x',x_n)\in B_{\lambda R}(0)\,:\,x_n<0,
\end{aligned}
\right.
\end{align*}
which implies
\begin{align*}
\nabla \overline\bfu(t,x)&:=\left\{\begin{aligned}
\nabla\tilde\bfu(t,x) &&\textrm{for } (x',x_n)\in B_{\lambda R}(0)\,:\,x_n\geq 0,\\
-\nabla \tilde\bfu(t,x',-x_n) \bfR &&\textrm{for }(x',x_n)\in B_{\lambda R}(0)\,:\,x_n<0.
\end{aligned}
\right.
\end{align*}
Similarly, we define $\overline{\bff}$ and $\overline{\bfu}_0$.
Finally, we reflect the elliptic part of our system by setting
\begin{align*}
\overline\bfS(t,x,\bfxi)&:=\left\{\begin{aligned}
\tilde\bfS(t,x,\bfxi) &&\textrm{for } (x',x_n)\in B_{\lambda R}(0)\,:\,x_n\geq 0,
\\
\tilde\bfS(t,(x',-x_n),\bfxi)   &&\textrm{for } (x',x_n)\in B_{\lambda R}(0)\,:\,x_n<0.
\end{aligned}
\right.
\end{align*}
We claim that the function $\overline\bfu$ is a weak solution to the system
\begin{align}
\label{eq:boundaryfinal}
\partial_t\overline\bfu=\divergence(\overline\bfS(\nabla \overline\bfu))+\overline\bff\text{ in }B_{\lambda R}(0)
\end{align}
with initial datum $\overline\bfu_0$, cf. Definition \ref{def:weak}.
For a given $\bfphi\in C^{0,1}_0(B_{\lambda R})$, we may split
 \[
 \bfphi(x',x_n)=\frac{\bfphi(x',x_n)+\bfphi(x',-x_n)}{2}+\frac{\bfphi(x',x_n)-\bfphi(x',-x_n)}{2}=:\bfphi_1(x',x_n)+\bfphi_2(x',x_n).
 \]
 This symmetry implies that
 \[
 \nabla \bfphi_1(x',x_n)= \nabla \bfphi_1(x',-x_n)\bfR\text{ and } \nabla \bfphi_2(x',x_n)= -\nabla \bfphi_2(x',-x_n)\bfR
 \]
and that $\bfphi_2\in C^{0,1}_0(B^+_{\lambda R})$. By a simple change of coordinates on the domain $B^-_{\lambda R}$ we obtain
\begin{align*}
&\int_{B_{\lambda R}}\big(\overline\bfu-\overline\bfu_0\big)\cdot\bfvarphi\dx+\int_0^t\int_{B_{\lambda R}} \overline\bfS(\sigma,x,\nabla\overline\bfu):\nabla\bfvarphi\dxs
-\int_0^t\int_{B_{\lambda R}} \overline\bff  \cdot\bfvarphi\dxs\\
&=2\int_{B_{\lambda R}^+}\big(\tilde\bfu-\tilde\bfu_0\big)\cdot\bfvarphi_2\dx+2\int_0^t\int_{B_{\lambda R}^+} \tilde\bfS(\sigma,x,\nabla\tilde\bfu):\nabla\bfvarphi_2\dxs
-2\int_0^t\int_{B_{\lambda R}^+} \tilde\bff  \cdot\bfvarphi_2\dxs=0
\end{align*}
using \eqref{eq:boundaryflat} and $\bfphi_2\in C^{0,1}_0(B^+_{\lambda R})$. That is, we have \eqref{eq:boundaryfinal}.
Although the new system \eqref{eq:boundaryfinal} is not exactly the $p(t,x)$-Laplacian, the growth conditions \eqref{growth1} and \eqref{growth2}
 still hold with $p$ replaced by
\begin{align*}
\overline p(t,x)&:=\left\{\begin{aligned}
&\tilde p(t,x) &&\textrm{for } (x',x_n)\in B_{\lambda R}(0)\,:\,x_n\geq 0,\\
&\tilde p(t,(x',-x_n)) &&\textrm{for }(x',x_n)\in B_{\lambda R}(0)\,:\,x_n<0.
\end{aligned}
\right.
\end{align*}
This is obvious as far as \eqref{growth1} is concerned. To verify \eqref{growth2}
we only have to check the case $x\in B_{\lambda R}^+$ and $y\in B_{\lambda R}^-$, recall \eqref{eq:1610}. 
By definition of $\overline\bfS$ and \eqref{eq:1610} we obtain
\begin{align*}
&|\overline\bfS(t,x,\bfeta)-\overline\bfS(t,y,\bfeta)|=\,|\tilde\bfS(t,x,\bfeta)-\tilde\bfS(t,(y',-y_n),\bfeta)| \\
&\leq \,c\,|(x',x_n)-(y',-y_n)|^{\alpha_x}\big(1+|\ln(\kappa+|\bfeta|)|\big)\Big( (\kappa+|\bfeta|)^{\tilde p(t,x)-2}+(\kappa+|\bfeta|)^{\tilde p(t,(y',-y_n))-2}\Big) |\bfeta|\\
&\leq \,c\,|x-y|^{\alpha_x}\big(1+|\ln(\kappa+|\bfeta|)|\big)\Big( (\kappa+|\bfeta|)^{\overline p(t,x)-2}+(\kappa+|\bfeta|)^{\overline p(t,y)-2}\Big) |\bfeta|.
\end{align*}
 Furthermore, due to the assumptions on the boundary, the regularity of the right-hand side and the initial datum remain (i.e., we have $\overline\bff\in \mathcal V^{\overline p'(\cdot)}$ because of the zero boundary conditions). Consequently, the interior regularity theory holds, in particular \eqref{eq:1902a} applies to \eqref{eq:boundaryfinal}, i.e., by setting
$
\overline\bfF(t,x,\bfxi)=(\kappa+|\bfxi|)^{\frac{\overline p(t,x)-2}{2}}\bfxi
$
we obtain 
\begin{align}\label{eq:1902aboundary}
\begin{aligned}
\sup_{t\in [0,T]}[\overline\bfu]^2_{W^{\tilde\alpha_x,2}(B_\rho)}&+\int_0^T[\overline\bfF(\cdot,\cdot,\nabla\overline\bfu)]^2_{W^{\tilde\alpha_x,2}(B_\rho)}\dt
\\&
\leq\,c_B\int_0^T\int_{10B_\rho} \big(1+|\nabla \overline\bfu|^{sp}\big)\dxt
\end{aligned}
\end{align}
for all $\tilde\alpha_x<\alpha_x$ provided $\rho$ is small enough say, $\rho \leq \lambda R/10$ . Consequently, we also have
\begin{align}\label{eq:1902aboundary'}
\begin{aligned}
\sup_{t\in [0,T]}[\tilde\bfu]^2_{W^{\tilde\alpha_x,2}(B_\rho^+)}&+\int_0^T[\tilde\bfF(\cdot,\cdot,\nabla\tilde\bfu)]^2_{W^{\tilde\alpha_x,2}(B_\rho^+)}\dt
\\&
\leq\,c_B\int_0^T\int_{10B_\rho^+} \big(1+|\nabla \tilde\bfu|^{sp}\big)\dxt
\bigg],
\end{aligned}
\end{align}
where $
\tilde\bfF(t,x,\bfxi)=(\kappa+|\bfxi|)^{\frac{\tilde p(t,x)-2}{2}}\bfxi
$. We would now like to derive a similar estimate for $\bfu$.
First of all, \eqref{eq:1902aboundary'} implies, by the use of \eqref{inclusion2} and the Lipschitz continuity of $\Psi^{-1}$, that
\begin{align*}
&[\tilde\bfF(\cdot,\cdot,\nabla\tilde\bfu)]^2_{W^{\tilde\alpha_x,2}(B_\rho^+)}=\int_{B_\rho^+}\int_{B_\rho^+}\frac{|\tilde\bfF(\cdot,y,\nabla\tilde\bfu(y))-\tilde\bfF(\cdot,x,\nabla\tilde\bfu(x))|^2}{|y-x|^{n+2\tilde\alpha_x }}\,\dd x\,\dd y\\
&=\int_{\Psi^{-1}(B_\rho^+)}\int_{\Psi^{-1}(B_\rho^+)}\frac{\big|\bfF\big(\cdot,y,\nabla\bfu(y)\bfJ^{-1}(y) \big) - \bfF \big(\cdot,x,\nabla\bfu(x)\bfJ^{-1}(x) \big) \big|^2}{|\Psi(y)-\Psi(x)|^{n+2\tilde\alpha_x }}\,\dd x\,\dd y\\
&\geq\,c\,\int_{B_{\varrho}\cap\Omega}\int_{B_{\varrho}\cap\Omega}\frac{ \big|\bfF\big(\cdot,y,\nabla\bfu(y)\bfJ^{-1}(y) \big) - \bfF\big(\cdot,x,\nabla\bfu(x)\bfJ^{-1}(x) \big) \big|^2}{|y-x|^{n+2\tilde\alpha_x }}\,\dd x\,\dd y
\end{align*}
using \eqref{inclusion2}, where $\varrho=\rho/\Lambda$. In the following we use the estimate
\begin{align}
\vert\bfF(t,x, \bfxi) - \bfF(t,y,\bfxi)\vert
&\leq 
c_s\, \vert y-x \vert^{\alpha_x}(1 + \vert \bfxi \vert)^{\frac{p(t,y)s}{2}}
\label{eq:2808'}
\end{align}
which holds for arbitrary $s>1$ provided $x,y\in B_{\varrho}\cap\Omega$ and $\rho$ is small enough. It can be shown exactly as in \eqref{eq:2808}. Under the same assumptions, we also have 
\begin{align}
\vert\bfF(t,y, \bfxi\bfJ^{-1}(y)) - \bfF(t,y,\bfxi\bfJ^{-1}(x))\vert
&\leq 
c_s\, \vert y-x \vert^{\alpha_x}(1 + \vert \bfxi \vert)^{\frac{p(t,y)s}{2}}
\label{eq:2808''}
\end{align}
by the use of the
H\"older-continuity of $\bfJ^{-1}$ and Lemma \ref{lem:hammer}.\\
Combining \eqref{eq:2808''} and \eqref{eq:2808'}, we obtain
\begin{align*}
&\int_{B_{\varrho}\cap\Omega}\int_{B_{\varrho}\cap\Omega}\frac{ \big| \bfF\big(\cdot,y,\nabla\bfu(y)\bfJ^{-1}(y) \big) - \bfF\big(\cdot,x,\nabla\bfu(x)\bfJ^{-1}(x) \big) \big|^2}{|y-x|^{n+2\tilde\alpha_x }}\,\dd x\,\dd y\\
&\geq \int_{B_{\varrho}\cap\Omega}\int_{B_{\varrho}\cap\Omega}\frac{ \big| \bfF\big(\cdot,x,\nabla\bfu(y)\bfJ^{-1}(x) \big) - \bfF\big(\cdot,x,\nabla\bfu(x)\bfJ^{-1}(x)\big) \big|^2}{|y-x|^{n+2\tilde\alpha_x }}\,\dd x\,\dd y\\
&-\int_{B_{\varrho}\cap\Omega}\int_{B_{\varrho}\cap\Omega}\frac{ \big| \bfF\big(\cdot,y,\nabla\bfu(y)\bfJ^{-1}(y) \big) - \bfF\big(\cdot,x,\nabla\bfu(y)\bfJ^{-1}(x) \big) \big|^2}{|y-x|^{n+2\tilde\alpha_x }}\,\dd x\,\dd y\\
&\geq\,c\,\int_{B_{\varrho}\cap\Omega}\int_{B_{\varrho}\cap\Omega}\frac{|\bfF(\cdot,x,\nabla\bfu(y))-\bfF(\cdot,x,\nabla\bfu(x))|^2}{|y-x|^{n+2\tilde\alpha_x }}\,\dd x\,\dd y\\
&-c\int_{B_{\varrho}\cap\Omega}\int_{B_{\varrho}\cap\Omega}\frac{(1+|\nabla\bfu(y)|)^{ sp(t,y)}}{|y-x|^{n+2(\tilde\alpha_x-\alpha_x) }}\,\dd x\,\dd y\\
&\geq\,c\,\int_{B_{\varrho}\cap\Omega}\int_{B_{\varrho}\cap\Omega}\frac{|\bfF(\cdot,y,\nabla\bfu(y))-\bfF(\cdot,x,\nabla\bfu(x))|^2}{|y-x|^{n+2\tilde\alpha_x }}\,\dd x\,\dd y\\
&-c\,\int_{B_{\varrho}\cap\Omega}\int_{B_{\varrho}\cap\Omega}\frac{|\bfF(\cdot,y,\nabla\bfu(y))-\bfF(\cdot,x,\nabla\bfu(y))|^2}{|y-x|^{n+2\tilde\alpha_x }}\,\dd x\,\dd y\\
&-c\int_{B_{\varrho}\cap\Omega}\int_{B_{\varrho}\cap\Omega}\frac{(1+|\nabla\bfu(y)|)^{sp(t,y)}}{|y-x|^{n+2(\tilde\alpha_x-\alpha_x) }}\,\dd x\,\dd y\\
&\geq\,c\,[\bfF(\cdot,\cdot,\nabla\bfu)]^2_{W^{\tilde\alpha_x,2}(B_\rho\cap\Omega)}-c\int_{B_{\varrho}\cap\Omega}\int_{B_{\varrho}\cap\Omega}\frac{(1+|\nabla\bfu(y)|)^{sp(t,y)}}{|y-x|^{n+2(\tilde\alpha_x-\alpha_x) }}\,\dd x\,\dd y.
\end{align*}
Note that we also took into account 
 \eqref{lL} in the second step.
Similarly, we obtain
$$[\tilde\bfu]^2_{W^{\tilde\alpha_x,2}(B_\rho^+)}\geq\,c\, [\bfu]^2_{W^{\tilde\alpha_x,2}(B_\varrho\cap\Omega)}.$$
By standard properties of the Riesz potential on bounded domains, we infer
 \begin{align*}
\int_{B_{\varrho}\cap\Omega}\int_{B_{\varrho}\cap\Omega}\frac{(1+|\nabla\bfu(y)|)^{sp(t,y)}}{|y-x|^{n+2(\tilde\alpha_x-\alpha_x) }}\,\dd x\,\dd y
&\leq\,c\,\int_{B_{\varrho}\cap\Omega}(1+|\nabla\bfu(y)|)^{sp(t,y)}\dd y
\end{align*}
provided that $\tilde\alpha_x<\alpha_x$.
Collecting the estimates above we conclude that
\begin{align}\label{eq:1902aboundary}
\begin{aligned}
\sup_{t\in [0,T]}[\bfu]^2_{W^{\tilde\alpha_x,2}(B_{ \varrho}\cap\Omega)}&+\int_0^T[\bfF(\cdot,\cdot,\nabla\bfu)]^2_{W^{\tilde\alpha_x,2}(B_{\varrho}\cap\Omega)}\dt
\\&
\leq\,c_B\,\int_0^T\int_{B_{10B_ \rho}\cap\Omega} \big(1+|\nabla \bfu|\big)^{sp}\dxt
\end{aligned}
\end{align}
for all $\tilde\alpha_x<\alpha_x$, where the balls are centered at $x_0$.
Now, we cover $\partial\Omega$ and $\Omega$ with a finite family of appropriate balls and gain from \eqref{eq:1902a} and \eqref{eq:1902aboundary}
\begin{align}\label{eq:spacefinal}
\begin{aligned}
\sup_{t\in [0,T]}[\bfu]^2_{W^{\tilde\alpha_x,2}(\Omega)}&+\int_0^T[\bfF(\cdot,\cdot,\nabla\bfu)]^2_{W^{\tilde\alpha_x,2}(\Omega)}\dt
\\&
\leq\,c\int_0^T\int_{\Omega} \big(1+|\nabla \bfu|^{sp}\big)\dxt.
\end{aligned}
\end{align}

\subsection{Regularity in time}
Now, we analyse the regularity with respect to time and prove
\begin{align}\label{eq:timefinal}
[\bfF(\cdot,\cdot,\nabla\bfu)]^2_{W^{\tilde\alpha_t,2}(I;L^2(\Omega))}
\leq\,c\int_0^T\int_{\Omega} \big(1+|\nabla \bfu|^{sp}\big)\dxt.
\end{align}
We start by showing that
\begin{align}\label{eq:1010}
\sup_{t>0}\int_{\Omega}\frac{|\bfu(t)-\bfu(0)|^2}{t^2}\dx\leq\,c.
\end{align}
Equation \eqref{eq:heatweak1} can be formulated as
\begin{align*}
\int_\Omega(\bfu-\bfu_0)\cdot\bfvarphi\dx&= -\int_0^t\int_\Omega \bfS(\cdot,\cdot,\nabla\bfu):\nabla\bfvarphi\dxs+\int_0^t\int_\Omega \bff\cdot\bfvarphi\dxs
\end{align*}
for all $\bfphi\in C^\infty_0(\Omega)$. Multiplying this by $\partial_t\psi$ with $\psi\in C^\infty_0([0,T))$ and integrating by parts we obtain
\begin{align*}
\int_Q(\bfu-\bfu_0)\cdot\bfvarphi\,\partial_t\psi\,\dxt&= \int_Q \bfS(\cdot,\cdot,\nabla\bfu):\nabla\bfvarphi\,\psi\dxt-\int_Q \bff\cdot\bfvarphi\,\psi\dxt
\end{align*}
for all $\bfphi\in C^\infty_0(\Omega)$ and $\psi\in C^\infty_0([0,T))$ and finally
\begin{align*}
\int_Q(\bfu-\bfu_0)\cdot\partial_t\bfPsi\,\dxt&= \int_Q \bfS(\cdot,\cdot,\nabla\bfu):\nabla\bfPsi\,\dxt-\int_Q \bff\cdot\bfPsi\dxt
\end{align*}
for all $\bfPsi\in C^\infty_0([0,T)\times\Omega)$ (using that test-functions which factorize in space and time are dense).
Setting $\bfPsi=(\bfu-\bfu_0)\chi_{(0,t)}$ (in fact, one has to use a smooth approximation) yields
\begin{align}\nonumber
\|\bfu(t)-\bfu_0\|_{L^2(\Omega)}^2&=-2\int_0^t\int_\Omega\Big(\bfS(\cdot,\cdot,\nabla \bfu)-\bfS(\cdot,\cdot,\nabla \bfu_0)\Big):\nabla\big( \bfu-\bfu_0\big)\,\dif x\,\dif\sigma\\
&\quad +2\int_0^t\int_\Omega\diver\bfS(\cdot,\cdot,\nabla \bfu_0)\cdot (\bfu-\bfu_0)\,\dif x\,\dif\sigma+2\int_0^t\int_\Omega \bff\cdot(\bfu-\bfu_0)\dxs.\label{eq:0603}
\end{align}
Now fix $t_0\in(0,T)$ and consider $t\in(0,t_0)$.
By H\"older's inequality and the assumptions on $\bfu_0$ and $\bff$, recall \eqref{ass:fu0}, we deduce that
\begin{align*}
&\|\bfu(t)-\bfu_0\|_{L^2(\Omega)}^2+2\,\int_0^t\int_\Omega\Big(\bfS(\cdot,\cdot,\nabla \bfu)-\bfS(\cdot,\cdot\nabla \bfu_0)\Big):\nabla\big( \bfu-\bfu_0\big)\,\dif x\,\dif\sigma\\
&\quad\leq \int_0^{t}\frac{1}{t_0}\|\bfu-\bfu_0\|_{L^2(\Omega)}^2\,\dif\sigma+ct_0^2\,\Big(\|\diver\bfS(0,\cdot,\nabla \bfu_0)\|_{L^2(\Omega)}^2+\sup_{t\in(0,T)}\|\bff\|^2_{L^2(\Omega)}\Big)\\
&\quad\leq \int_0^{t}\frac{1}{t_0}\|\bfu-\bfu_0\|_{L^2(\Omega)}^2\,\dif\sigma+ct_0^2,
\end{align*}
where $c$ is independent of $c_0$.
Finally, Gronwall's lemma implies
\begin{align*}
\|\bfu(t_0)-\bfu_0\|_{L^2(\Omega)}^2\leq \sup_{t\in(0,t_0)}\|\bfu(t)-\bfu_0\|_{L^2(\Omega)}^2\leq\,ct_0^2\exp\bigg(\int_0^{t_0}\frac{1}{t_0}\ds\bigg)\leq \,ct_0^2
\end{align*}
which completes the proof of \eqref{eq:1010} by arbitrariness of $t_0$.\\
Let us consider the difference $\bfw_h(t):=\bfu(t)-\bfu(t-h)$ for $t\in(h,T)$
and $0<h\ll1$ and argue as in \eqref{eq:0603}. We find that
\begin{align*}
&\|\bfw_h(t)\|_{L^2(\Omega)}^2
=\|\bfw_h(h)\|_{L^2(\Omega)}^2+2\int_h^t\int_\Omega\big(\bff(\sigma)-\bff(\sigma-h)\big) \cdot\bfw_h\dxs\\&-2\int_h^t\int_\Omega\Big(\bfS(\sigma,\cdot,\nabla \bfu(\sigma))-\bfS(\sigma-h,\cdot,\nabla \bfu(\sigma-h))\Big):\big(\nabla \bfu(\sigma)-\nabla \bfu(\sigma-h)\big)\,\dif x\,\dif\sigma\\
&=\|\bfw_h(h)\|_{L^2(\Omega)}^2+2\int_h^t\int_\Omega\big(\bff(\sigma)-\bff(\sigma-h)\big) \cdot\bfw_h\dxs\\&-2\int_h^t\int_\Omega\Big(\bfS(\sigma,\cdot,\nabla \bfu(\sigma))-\bfS(\sigma,\cdot,\nabla \bfu(\sigma-h))\Big):\big(\nabla \bfu(\sigma)-\nabla \bfu(\sigma-h)\big)\,\dif x\,\dif\sigma\\
&-2\int_h^t\int_\Omega\Big(\bfS(\sigma,\cdot,\nabla \bfu(\sigma-h))-\bfS(\sigma-h,\cdot,\nabla \bfu(\sigma-h))\Big):\big(\nabla \bfu(\sigma)-\nabla \bfu(\sigma-h)\big)\,\dif x\,\dif\sigma\\
&=:(I)+(II)+(III)+(IV).
\end{align*}
In view of \eqref{eq:1010} we obtain
\begin{align*}
(I)\leq\,ch^2.
\end{align*}
Using the assumptions on $\bff$ we have
\begin{align*}
(II)&\leq\,\int_h^t\int_\Omega|\bff(\sigma)-\bff(\sigma-h)|^2 \dxs+\int_h^t\int_\Omega|\bfw_h|^2\dxs\\
&\leq\,ch^{2\alpha_t}[\bff]^2_{C^{\alpha_t}(0,T;L^2(\Omega))}+\int_h^t\int_\Omega|\bfw_h|^2\dxs,
\end{align*}
where $[\cdot]_{C^{\alpha_t}(0,T;L^2(\Omega))}$ denotes the usual H\"older-seminorm.
Furthermore, if we define $\tau^\sigma_{-h} v(\sigma):=v(\sigma)-v(\sigma-h)$, we have
\begin{align*}
-(III)&\sim\, \int_h^t\int_{\Omega}(\kappa+|\nabla \bfu(\sigma,x)|+|\nabla\bfu(\sigma-h,x)|)^{p-2}|\tau_{-h}^\sigma\nabla\bfu|^2\dxs\\
&\geq\,\int_h^t\int_{\Omega}|\tau_{-h}^\sigma\bfF(\cdot,\cdot,\nabla\bfu)|^2\dxs-ch^{2\alpha_t}\int_{Q}(1+|\nabla\bfu|^{sp})\dxs
\end{align*}
as a consequence of Lemma \ref{lem:hammer} and the H\"older-continuity of $p$.
Finally, we get for $\varepsilon>0$ arbitrary
\begin{align*}
(IV)
&\leq\,\varepsilon \int_h^t\int_{\Omega} \big(\kappa+|\nabla\bfu(\sigma,x)|+|\nabla\bfu(\sigma-h,x)|\big)^{p-2}|\tau_{-h}^\sigma\nabla\bfu|^2\dxs\\
&+c(\varepsilon) \,h^{2\alpha_t }\int_h^t\int_{\Omega} \big(1+|\nabla\bfu(\sigma,x)|+|\nabla\bfu(\sigma-h,x)|\big)^{sp}\dxs.
\end{align*}
In the estimates for $(III)$ and $(IV)$, we argued similarly as in the proof of \eqref{eq:1902a} and in particular, used estimates in the spirit of \eqref{eq:hxS} and \eqref{eq:hxF} for time-differences.
Plugging all together, choosing $\varepsilon$ small enough and applying Gronwall's lemma, we have  shown that
\begin{align*}
 \int_h^T\int_\Omega|\bfF(\sigma,x,\nabla \bfu(\sigma))&-\bfF(\sigma-h,x,\nabla \bfu(\sigma-h))|^2\,\dxs\\&\leq \,ch^{2\alpha_t}\int_0^T\int_{\Omega}\big(1+ |\nabla \bfu|^{sp}\big)\dxs
\end{align*}
or equivalently 
\begin{align}\label{eq:neu}
\begin{aligned}
\,\int_0^{T-h} \int_\Omega|\bfF(\sigma+h,x,\nabla \bfu(\sigma+h))&-\bfF(\sigma,x,\nabla \bfu(\sigma))|^2\dxs\\&\leq ch^{2\alpha_t}\,\int_0^T\int_{\Omega}\big(1+ |\nabla \bfu|^{sp}\big)\dxs
\end{aligned}
\end{align}
which implies \eqref{eq:timefinal}.\\\

\subsection{Conclusion}


Combining \eqref{eq:spacefinal} and \eqref{eq:timefinal} implies for $\alpha<\min\{\alpha_x,\alpha_t\}$
\begin{align}\label{eq:1902c}
\begin{aligned}
\|\bfF(\nabla\bfu)\|_{W^{\alpha,2}(Q)}^2&\leq\,c\,\bigg(\int_{Q} |\nabla \bfu|^{sp}\dxt+1\bigg)\\
&\leq\,c\,\bigg(\int_{Q} |\bfF(\nabla \bfu)|^{2+\varepsilon}\dxt+1\bigg).
\end{aligned}
\end{align}
 Here $\varepsilon>0$ can be chosen arbitrarily close to zero provided $s$ is close enough to one. Moreover, we have $W^{\alpha,2}(Q)\hookrightarrow L^{\chi}(Q)$ for some $\chi=\chi(n,\alpha)>2$, cf. Lemma \ref{lema}. On the other hand, we can interpolate $L^{2+\varepsilon}(Q)$ between $L^2(Q)$ and $L^\chi(Q)$. This implies
\begin{align*}
\int_{Q} |\bfF(\cdot,\cdot,\nabla \bfu)|^{2+\varepsilon}\dxt&\leq\bigg(\int_{Q} |\bfF(\cdot,\cdot,\nabla \bfu)|^{2}\dxt\bigg)^{\frac{2+\varepsilon}{2}\theta}\bigg(\int_{Q} |\bfF(\nabla \bfu)|^{\chi}\dxt\bigg)^{\frac{2+\varepsilon}{\chi}(1-\theta)}\\
&\leq\,c\,\|\bfF(\cdot,\cdot,\nabla\bfu)\|_{W^{\alpha,2}(Q)}^{(1-\theta)(2+\varepsilon)},\quad \theta=\Big(\frac{2\chi}{2+\varepsilon}-2\Big)\frac{1}{\chi-2}.
\end{align*}
In the above we used $\bfF(\cdot,\cdot,\nabla\bfu)\in \big(L^2(Q)\big)^{N\times n}$ which follows directly from $\bfu\in\big(\mathcal V^{p(\cdot)}(Q)\big)^N$ and the definition of $\bfF$.
If $\varepsilon$ is small enough, we have $(1-\theta)(2+\varepsilon)=\varepsilon\frac{\chi}{\chi-2}<1$. Hence, as a consequence of Young's inequality,
\begin{align*}
\int_{Q} |\bfF(\cdot,\cdot,\nabla \bfu)|^{2+\varepsilon}\dxt
&\leq\,\delta\,\|\bfF(\cdot,\cdot,\nabla\bfu)\|_{W^{\alpha,2}(Q)}^2+c_\delta
\end{align*}
for any $\delta>0$. Inserting this into \eqref{eq:1902c} yields for $\delta$ small enough
\begin{align}\label{eq:1902'}
\|\bfF(\nabla\bfu)\|_{W^{\alpha,2}(Q)}^2&\leq\,c.
\end{align}
Note that this implies by Sobolev's embedding that $\bfF(\cdot,\cdot,\nabla\bfu)\in \big(L^\chi(Q)\big)^{N\times n}$.
 Consequently, the term on the right-hand side of
 \eqref{eq:spacefinal} and \eqref{eq:timefinal} is bounded. The claim follows.
\hfill $\Box$

\appendix

\section{\texorpdfstring{Orlicz spaces}{Orlicz spaces}}
\label{sec:Orlicz spaces}
\noindent
The following definitions and results are standard in the theory of
Orlicz spaces and can for example be found in~\cite{RaoR91}.
A continuous, convex function $\rho\,:\, [0,\infty) \to [0,\infty)$
with $\rho(0)=0$, and $\lim_{t \to \infty} \rho(t) = \infty$ is called
a {\em continuous, convex $\phi$-function}.  

We say that $\phi$
satisfies the $\Delta_2$--condition, if there exists $c > 0$ such that
for all $t \geq 0$ holds $\phi(2t) \leq c\, \phi(t)$. By
$\Delta_2(\phi)$ we denote the smallest such constant. Since $\phi(t)
\leq \phi(2t)$ the $\Delta_2$-condition is equivalent to $\phi(2t)
\sim \phi(t)$ uniformly in $t$. For a family $\phi_\lambda$ of
continuous, convex $\phi$-functions we define
$\Delta_2(\set{\phi_\lambda}) := \sup_\lambda \Delta_2(\phi_\lambda)$.
Note that if $\Delta_2(\phi) < \infty$ then $\phi(t) \sim \phi(c\,t)$
uniformly in $t\geq 0$ for any fixed $c>0$.
By $L^\phi$ and $W^{k,\phi}$, $k\in \setN_0$,  we denote the classical Orlicz and
Orlicz-Sobolev spaces, i.e.\ $f \in L^\phi$ iff $\int
\phi(\abs{f})\dx < \infty$ and $f \in W^{k,\phi}$ iff $ \nabla^j f
\in L^\phi$, $0\le j\le k$.\\
A $\phi$-function $\rho$ is called a $N$-function iff it is strictly increasing and convex with
\begin{align*}
  \lim_{t\rightarrow0}\frac{\rho(t)}{t}=
  \lim_{t\rightarrow\infty}\frac{t}{\rho(t)}=0.
\end{align*}
By $\rho^*$ we denote the conjugate N-function of $\rho$, which is
given by $\rho^*(t) = \sup_{s \geq 0} (st - \rho(s))$. Then $\rho^{**}
= \rho$.
\begin{lemma}[Young's inequality]
  \label{lem:young}
  Let $\rho$ be an N-function. Then for all $s,t\geq 0$ we have
  \begin{align*}
    st \leq \rho(s)+\rho^*(t).
  \end{align*}
  If $\Delta_2(\rho,\rho^*)< \infty$, then additionally for all $\delta>0$
  \begin{align*}
    st &\leq \delta\,\rho(s)+c_\delta\,\rho^*(t),
    \\
    st &\leq c_\delta\,\rho(s)+\delta\,\rho^*(t),
\\
\rho'(s)t &\leq \delta\,\rho(s)+c_\delta\,\rho(t),
\\
\rho'(s)t &\leq \delta\,\rho(t)+c_\delta\,\rho(s),
  \end{align*}
  where $c_\delta= c(\delta, \Delta_2(\set{\rho,\rho^*}))$.
\end{lemma}

\begin{definition}
  \label{ass:phipp}
  Let $\rho$ be an N-function. 
  We say that $\rho$ is {\em elliptic}, if
  $\rho$ is $C^1$ on $[0,\infty)$ and $C^2$ on $(0,\infty)$ and
  assume that 
  \begin{align}
    \label{eq:phipp}
    \rho'(t) &\sim t\,\rho''(t)
  \end{align}
  uniformly in $t > 0$. The constants hidden in $\sim$ are called the
  {\em characteristics of~$\rho$}.
\end{definition}
Note that~\eqref{eq:phipp} is stronger than
$\Delta_2(\rho,\rho^*)<\infty$. In fact, the $\Delta_2$-constants can
be estimated in terms of the characteristics of~$\rho$.

Associated to an elliptic $N$-function $\rho$ we define the tensors
\begin{align*}
  \bfS^\rho(\bfxi)&:=\frac{\rho'(\abs{\bfxi})}{\abs{\bfxi}}\bfxi,\quad
  \bfxi\in\mathbb R^{n\times n}
  \\
  \bfF^\rho(\bfxi)&:=\sqrt{\frac{\rho'(\abs{\bfxi})}{\abs{\bfxi}}}\,\bfxi,\quad
  \bfxi\in\mathbb R^{n\times n}.
\end{align*}

We define the {\em shifted} $N$-function $\rho_a$ for $a\geq 0$ by
\begin{align}
  \label{eq:def_shift}
  \rho_a(t) &:= \int_0^t \frac{\rho'(a+\tau)}{a+\tau} \tau\,\dd\tau.
\end{align}

The following auxiliary result can be found in~\cite{DieE08,DR}.
\begin{lemma}
  \label{lem:shift_sim}
  For all $a,b, t \geq 0$ we have
  \begin{align*}
    \rho_a(t) &\sim 
    \begin{cases}
      \rho''(a) t^2 &\qquad\text{if $t \lesssim a$}
      \\
      \rho(t) &\qquad\text{if $t \gtrsim a$,}
    \end{cases}
\\
(\rho_a)_b(t)&\sim\rho_{a+b}(t).
  \end{align*}
\end{lemma}
\begin{lemma}[{\cite[Lemma~2.3]{DieE08}}]
  \label{lem:hammer}
  We have
  \begin{align*}
    \begin{aligned}
      \big({\bfS^\rho}(\bfP) - {\bfS^\rho}(\bfQ)\big) \cdot
      \big(\bfP-\bfQ \big) &\sim \bigabs{ \bfF^\rho(\bfP) -
        \bfF^\rho(\bfQ)}^2
      \\
      &\sim \rho_{\abs{\bfP}}(\abs{\bfP - \bfQ})
      \\
      &\sim \rho''\big( \abs{\bfP} + \abs{\bfQ}
      \big)\abs{\bfP - \bfQ}^2
    \end{aligned}
  \end{align*}
  uniformly in $\bfP, \bfQ \in \setR^{N \times n}$.  Moreover,
  uniformly in $\bfQ \in \setR^{N \times n}$,
  \begin{align*}
    \bfS^\rho(\bfQ) \cdot \bfQ &\sim \abs{\bfF^\rho(\bfQ)}^2\sim
    \rho(\abs{\bfQ})\\
  \abs{{\bfS^\rho}(\bfP) - {\bfS^\rho}(\bfQ)}&\sim\big(\rho_{\abs{\bfP}}\big)'(\abs{\bfP - \bfQ}).
  \end{align*}
  The constants depend only on the characteristics of $\rho$.
\end{lemma}
The following inequality follows from immediately from Lemmas \ref{lem:young} and \ref{lem:hammer}.
\begin{lemma}
  \label{lem:young2}
  Let $\rho$ be an elliptic N-function. Then for each $\delta>0$ there
  exists $C_\delta \geq 1$ (only depending on~$\delta$ and the
  characteristics of~$\rho$) such that
  \begin{align*}
    \big({\bfS^\rho}(\bfP) - {\bfS^\rho}(\bfQ)\big) \cdot
      \big(\bfR-\bfQ \big) &\leq \delta\bigabs{ \bfF^\rho(\bfP) -
        \bfF^\rho(\bfQ)}^2+C_\delta \bigabs{ \bfF^\rho(\bfR) -
        \bfF^\rho(\bfQ)}^2
  \end{align*}
  for all $\bfP,\bfQ,\bfR\in\setR^{N\times n}$ and $t\geq0$.
\end{lemma}

\begin{lemma}[Change of Shift]
  \label{lem:shift_ch}
  Let $\rho$ be an elliptic N-function. 
    \begin{enumerate}
  \item[(a)] Then for each $\delta>0$ there
  exists $C_\delta \geq 1$ (only depending on~$\delta$ and the
  characteristics of~$\rho$) such that
  \begin{align*}
    \rho_{\abs{\bfP}}(t)&\leq C_\delta\, \rho_{\abs{\bfQ}}(t)
    +\delta\, \rho_{\abs{\bfP}}(\abs{\bfP - \bfQ}),
    \\
    (\rho_{\abs{\bfP}})^*(t)&\leq C_\delta\, (\rho_{\abs{\bfQ}})^*(t)
    +\delta\, \rho_{\abs{\bfP}}(\abs{\bfP - \bfQ}),
  \end{align*}
  for all $\bfP,\bfQ\in\setR^{N\times n}$ and $t\geq0$.
  \item[(b)] There
  exists $c \geq 1$ (only depending on the
  characteristics of~$\rho$) such that
  \begin{align*}
  \rho'_{|\bfP|}(|\bfQ-\bfP|)\leq\,c\,\rho'_{|\bfR|}(|\bfQ-\bfR|)+c\,\rho'_{|\bfR|}(|\bfP-\bfR|)
  \end{align*}
  for all $\bfP,\bfQ,\bfR\in\setR^{N\times n}$.
  \end{enumerate}
\end{lemma}
The case $\bfa=0$ or $\bfb=0$ in part (a) implies the following corollary.
\begin{corollary}[Removal of Shift]
  \label{cor:shift_ch}
  Let $\rho$ be an elliptic N-function. 
Then for each $\delta>0$ there
  exists $C_\delta \geq 1$ (only depending on~$\delta$ and the
  characteristics of~$\rho$) such that
  \begin{align*}
    \rho_{\abs{\bfa}}(t)&\leq C_\delta\, \rho(t) +\delta\,
    \rho(\abs{\bfa}),
    \\
    \rho(t)&\leq C_\delta\, \rho_{\abs{\bfa}}(t) +\delta\,
    \rho(\abs{\bfa}),
  \end{align*}
  for all $\bfa\in\setR^d$ and $t\geq0$.
\end{corollary}
\begin{lemma}
  \label{lem:shifted2}
  Let $\rho$ be an elliptic N-function. Then $(\rho_a)^*(t) \sim
  (\rho^*)_{\rho'(a)}(t)$ uniformly in $a,t \geq 0$. Moreover, for all
  $\lambda \in [0,1]$ we have
  \begin{align*}
    \rho_a(\lambda a) &\sim \lambda^2 \rho(a) \sim
    (\rho_a)^*(\lambda \rho'(a)).
  \end{align*}
  Finally, we have $\rho_a^\ast(\rho'_a(t))\sim \rho_a(t)$ uniformly in $a,t\geq0$.
\end{lemma}
\begin{lemma}
  \label{lem:shiftedindex}
  Let $\rho(t) := \int_0^t (\kappa+s)^{q-2} s\,\dd s$ with $q \in
  (1,\infty)$ and $t\geq 0$. Then 
  \begin{align*}
    \rho_a(\lambda t) &\leq c\, \max\set{\lambda^q, \lambda^2}
    \rho_a(t),
    \\
    (\rho_a)^*(\lambda t) &\leq c\, \max\set{\lambda^{q'}, \lambda^2}
    (\rho_a)^*(t)
  \end{align*}
  uniformly in $a,\lambda \geq 0$.
\end{lemma}
\begin{remark}
\label{A9}
  Let $p\in\mathcal P(Q)$ with $p^{-}>1$ and $p^+<\infty$.  The
  results above extend to the function $\phi^{(\sigma,x)}(t)=\int_0^t
  (\kappa+s)^{p(\sigma,x)-2}s\,\dd s$ uniformly in $x\in\Omega$, where the  constants only depend on $p^-$ and $p^+$.
\end{remark}
\begin{lemma}\label{lemma:growth}
Let $p\in C^{\alpha_x,\alpha_t}(\overline{Q})$ with $p^->1$ and $p^+<\infty$. Let $\bfS:Q\times\R^{N\times n}\rightarrow \R^{N\times n}$ be given by
$$\bfS(t,x,\bfeta)=(\kappa+|\bfeta|)^{p(t,x)-2}\bfeta$$
for some $\kappa\geq0$. Then we have
\begin{align}
\lambda (\kappa+|\bfeta|)^{p(t,x)-2}|\bfxi|^2&\leq D_\bfeta\bfS(t,x,\bfeta)(\bfxi,\bfxi)\leq\,\Lambda (\kappa+|\bfeta|)^{p(t,x)-2}|\bfxi|^2,\label{growth1'}\\
|\bfS(t,x,\bfeta)-\bfS(s,y,\bfeta)|&\leq\,c\,(|t-s|^{\alpha_t}+|x-y|^{\alpha_x})\big(1+|\ln(\kappa+|\bfeta|)|\big)
\label{growth2'}\\
&\quad\times \Big( (\kappa+|\bfeta|)^{p(t,x)-2} + (\kappa+|\bfeta|)^{ p(s,y)-2} \Big) |\bfeta| ,\nonumber
\end{align}
for all $(t,x),(s,y)\in \overline{Q}$ and all $\bfeta,\bfxi\in\R^{N\times n}$ with positive constants $\lambda,\Lambda,c$ and $\kappa\geq0$.
\end{lemma}
\begin{proof}
Let $\varphi(t,x,\eta)=\int_0^\eta (\kappa+\xi)^{p(t,x)-2}\xi d\xi$. Then we have
\begin{align*}
\bfS(t,x,\bfeta)&=D_{\bfeta} \varphi(t,x,|\bfeta|)=\frac{\partial_\eta\varphi(t,x,|\bfeta|)}{|\bfeta|}\bfeta,\\
 D_{\bfeta} \bfS(t,x,\bfeta)&=D^2_{\bfeta} \varphi(t,x,|\bfeta|)\\&=\frac{\partial_\eta^2\varphi(t,x,|\bfeta|)|\bfeta|-\partial_\eta\varphi(t,x,|\bfeta|)}{|\bfeta|}\frac{\bfeta\otimes\bfeta}{|\bfeta|^2}+\frac{\partial_\eta\varphi(t,x,|\bfeta|)}{|\bfeta|}I.
\end{align*}
Using that
\begin{align*}
\partial_\eta\varphi(t,x,\eta)&=(\kappa+\eta)^{p(t,x)-2}\eta,\\
\partial_\eta^2\varphi(t,x,\eta)&=(\kappa+\eta)^{p(t,x)-2}+(p(t,x)-2)(\kappa+\eta)^{p(t,x)-3}\eta\\
&=(\kappa+\eta)^{p(t,x)-2}\Big(1+(p(t,x)-2)\frac{\eta}{\kappa+\eta}\Big)
\end{align*}
as well as boundedness of $p$ from below and above we obtain
\begin{align*}
\lambda (\kappa+|\bfeta|)^{p(t,x)-2}|\bfxi|^2&\leq D_{\bfeta}\bfS(t,x,\bfeta)(\bfxi,\bfxi)\leq\,\Lambda (\kappa+|\bfeta|)^{p(t,x)-2}|\bfxi|^2.
\end{align*}
Furthermore, we have
\begin{align*}
&\bfS(s,y,\bfeta)-\bfS(t,x,\bfeta)=(\kappa+|\bfeta|)^{p(s,y)-2}\bfeta-(\kappa+|\bfeta|)^{p(t,x)-2}\bfeta\\
&=\int_0^1\frac{d}{d\theta}(\kappa+|\bfeta|)^{(p(t,x)+\theta(p(s,y)-p(t,x))-2}\dd\theta\bfeta\\
&=\ln(\kappa+|\bfeta|)(p(s,y)-p(t,x))\int_0^1(\kappa+|\bfeta|)^{(p(s,y)+\theta(p(t,y)-p(s,y))-2}\dd\theta\bfeta\\
&=\ln(\kappa+|\bfeta|)(p(s,y)-p(t,x))\int_0^1(\kappa+|\bfeta|)^{p(s,y)+\theta(p(t,x)-p(s,y))}\dd\theta\frac{\bfeta}{(\kappa+|\bfeta|)^2}.
\end{align*}
If $\kappa+|\bfeta|\leq 1$ we estimate further
\begin{align*}
&|\bfS(s,y,\bfeta)-\bfS(t,x,\bfeta)|\\
&\leq\,c\,(|t-s|^{\alpha_t}+|x-y|^{\alpha_x})\ln(\kappa+|\bfeta|)(\kappa+|\bfeta|)^{\min\{p(t,x),(p(s,y)\}}\frac{|\bfeta|}{(\kappa+|\bfeta|)^2}\\
&\leq\,c\,(|t-s|^{\alpha_t}+|x-y|^{\alpha_x})\big(1+|\ln(\kappa+|\bfeta|)|\big)\\
&\quad\times \Big( (\kappa+|\bfeta|)^{ p(t,x)-2} + (\kappa+|\bfeta|)^{p(s,y)-2} \Big) |\bfeta|
\end{align*}
using H\"older continuity of $p$. If $\kappa+|\bfeta|\geq 1$ one has to replace $\min\{p(t,x),(p(s,y)\}$ by $\max\{p(t,x),(p(s,y)\}$ but the final estimate remains the same.
\end{proof}

\centerline{\bf Acknowledgement}
\noindent{
The authors would like to thank the referee for the careful reading of the manuscript and the valuable suggestions.}\\\

\bibliographystyle{spmpsci}
\bibliography{numerics}

\end{document}